\documentclass[reqno,11pt]{amsart}

\usepackage{amsmath} 
\usepackage{amssymb} 
\usepackage{amsfonts}
\usepackage{mathrsfs}
\usepackage{amsthm} 
\usepackage[all]{xy}
\usepackage{vmargin}       
\usepackage{lineno,hyperref}

\setmarginsrb{2cm}{1.5cm}{2cm}{3cm}{3cm}{1cm}{5cm}{1cm}     %%%%%
\newcommand{\Z}{\mathbb{Z}}

\newcommand{\R}{\mathbb{R}}

\newcommand{\characteristic}{\operatorname{char}} 
 
\newcommand{\rad}{\operatorname{rad}} 
\newcommand{\soc}{\operatorname{soc}} 
 
\newcommand{\im}{\operatorname{Im}} 
\newcommand{\ch}{\operatorname{ch}} 
\newcommand{\m}{\operatorname{m}} 
 
\newcommand{\SymP}{\operatorname{Sym}}  
\newcommand{\rank}{\operatorname{rank}} 
\newcommand{\Aff}{\operatorname{Aff}} 
\newcommand{\Ext}{\operatorname{Ext}} 
\newcommand{\Hom}{\operatorname{Hom}} 
\newcommand{\ind}{\operatorname{ind}}

\newcommand{\SL}{\operatorname{SL}} 
\newcommand{\SO}{\operatorname{SO}} 
\newcommand{\Spin}{\operatorname{Spin}} 

\def\bigquotient#1#2{%
    \left.\raise1ex\hbox{$#1$}\middle/\lower1ex\hbox{$#2$}\right.%
}
\def\quotient#1#2{%
    \left.\raise0.3ex\hbox{$#1$} \middle/\lower0.3ex\hbox{$#2$}\right.%
}

\newtheoremstyle{thme}
{2}						%Space above
{2}						%Space below
{\itshape}				%Body font
{}						%Indent amount (empty = no indent, \parindent = para indent)
{\bfseries}				%Thm head font
{}						%Punctuation after thm head
{\newline}				%Space after thm head: " " = normal interword space; \newline = linebreak
{}						%Thm head spec (can be left empty, meaning `normal')

\newtheoremstyle{rem}	% name
{2}						%Space above
{2}						%Space below
{\upshape}				%Body font
{}						%Indent amount (empty = no indent, \parindent = para indent)
{\bfseries}				%Thm head font
{}						%Punctuation after thm head
{\newline}				%Space after thm head: " " = normal interword space; \newline = linebreak
{}						%Thm head spec (can be left empty, meaning `normal')

\theoremstyle{thme}
\newtheorem{thmintro}{Theorem}
 
\newtheorem{corolintro}[thmintro]{Corollary} 
\newtheorem{propintro}[thmintro]{Proposition} 

\newtheorem{thm}{Theorem}[section]
\newtheorem*{thm*}{Theorem}

\newtheorem{prop}[thm]{Proposition}
\newtheorem{corol}[thm]{Corollary}
\newtheorem{lem}[thm]{Lemma}

\theoremstyle{rem}
\newtheorem*{remarksintro}{Remarks}

\newtheorem{remark}[thm]{Remark}

\newtheorem{defin}[thm]{Definition}

\usepackage{amsmath,amssymb,graphicx,verbatim}

\begin{document}     %%%%%%%%%%%%%%%%%%%%%%%%%%%%%%%%%%%%%%%%%%%%%%%%%%%%%%%%%%%%%%%%%%%%%%%%%%%%%%%%%%%%%%%%%%%%%%%%%%%%%%%%%%%%%%%%%%%%%%%%%%%%%%%%%%%%
%%%%%%%%%%%%%%%%%%%%%%%%%%%%%%%%%%%%%%%%%%%%%%%%%%%%%%%%%%%%%%%%%%%%%%%%%%%%%%%%%%%%%%%%%%%%%%%%%%%%%%%%%%%%%%%%%%%%%%%%%%%%%%%%%%%%%%%%%%%%%%%%%%%%%%%%%
%%%%%%%%%%%%%%%%%%%%%%%%%%%%%%% 
%%%%%%%%%%%%%%%%%%%%% 
%%%%%%%%%%%
% 

%%%%%%%%%%%%%%%%%%%%%%%%%%%%%%
%%%%%     Front Page     %%%%% 
%%%%%%%%%%%%%%%%%%%%%%%%%%%%%%

\title{Structure of certain Weyl modules for the Spin groups}
\author{Mika\"el Cavallin}
\address{Fachbereich Mathematik, Postfach 3049, 67653 Kaiserslautern, Germany.} 
\email{cavallin.mikael@gmail.com} 
\thanks{The author would like to acknowledge the support of the Swiss National Science Foundation through grants no. 20020-135144 as well as the ERC Advanced Grant through grants no. 291512.}

%%%%%%%%%%%%%%%%%%%%%%%%%%
%     ABSTRACT     %%%%%%%     
%%%%%%%%%%%%%%%%%%%%%%%%%%
%%%%%%%%%%%%%%%%%%%%%%%%%%
\begin{abstract}     %%%%%    
%%%%%%%%%%%%%%%%%%%%%%%%%%
%%%%%%%%%%%%%%%%%%%%%%%%%%

Let $K$ be an algebraically closed field of characteristic $p\geqslant 0$ and let $W$ be a finite-dimensional $K$-vector space of dimension greater than or equal to $5.$ In this paper, we give the structure of certain Weyl modules for $G=\Spin(W)$ in the case where $p\neq 2,$ as well as the dimension of the corresponding irreducible, finite-dimensional, rational $KG$-modules. In addition, we determine the composition factors of the restriction of certain irreducible, finite-dimensional, rational $K\SL(W)$-modules to $\SO(W).$  
\end{abstract}

%\begin{keyword}
%Algebraic groups, classical groups, representation theory, weight multiplicities, irreducible modules, composition factors, restriction rules.  
%\end{keyword}

\maketitle

%%%%%%%%%%%%%%%%%%%%%%%%%%%%%%%%
%%%%%     Line Numbers     %%%%%
%%%%%%%%%%%%%%%%%%%%%%%%%%%%%%%%
%\linenumbers              %%%%% 
%\modulolinenumbers[5]     %%%%%         
%%%%%%%%%%%%%%%%%%%%%%%%%%%%%%%%

%%%%%%%%%%%%%%%%%%%%%%%%%%%%%%%%%%
%----- Introduction     -----%%%%%
%%%%%%%%%%%%%%%%%%%%%%%%%%%%%%%%%%
\section{Introduction}     %%%%%%%   
%%%%%%%%%%%%%%%%%%%%%%%%%%%%%%%%%%
%%%%%%%%%%%%%%%%%%%%%%%%%%%%%%%%%%

Let $K$ be an algebraically closed field of characteristic $p\geqslant 0,$ and let $G$ be a simply connected, simple algebraic group over $K.$ Fixing a Borel subgroup $B$ of $G$ containing a maximal torus $T$ of $G,$ one obtains an associated set of dominant weights for $T,$ denoted by $X^+(T).$ It is well-known that for each $\varpi\in X^+(T),$ there exists a unique (up to isomorphism) finite-dimensional, irreducible, rational $KG$-module $L_G(\varpi)$ having highest weight $\varpi.$ In other words, the isomorphism classes of finite-dimensional, irreducible, rational modules for $G$ are in one-to-one correspondence with the aforementioned dominant weights for $T.$  
\vspace{5mm}

In characteristic zero, the dimension of each $L_G(\varpi)$ is known, and is given by the well-known Weyl's degree formula \cite[Corollary 24.3]{Humphreys1}. Also weight multiplicities in $L_G(\varpi)$ can be recursively computed using Freudenthal's formula \cite{Freudenthal}, or one of the many variants developed in the last decades. (We refer the reader to \cite{MP}, \cite{Bremner}, \cite{Racah}, \cite{Sahi}, \cite{LePa}, \cite{Sch}, or \cite{Cavallin} for a few examples.)  Closed formulas can also be used to obtain information on weight multiplicities, or even on the so-called character of a given irreducible module (see \cite{Racah} or \cite{Kostant}, for instance). Observe, however, that those methods   are often quite demanding in terms of complexity.
\vspace{5mm}

In positive characteristic, not much is known about irreducible $KG$-modules in general. However, following the construction in \cite[Section 2]{St1}, one obtains a universal highest weight module $V_G(\varpi)$ of highest weight $\varpi,$ for every $\varpi\in X^+(T),$ by finding an appropriate $\mathbb{Z}$-form in a suitable irreducible module for the corresponding complex Lie algebra, and then tensoring it by $K.$ The $KG$-module $V_G(\varpi)$ is called the \emph{Weyl module} of highest weight $\varpi,$ and has the property that its quotient by its unique maximal   submodule $\rad (\varpi)$ is irreducible with highest weight $\varpi.$ In other words, we have
$$
L_G(\varpi) \cong \bigquotient{V_G(\varpi)}{\rad (\varpi)}.
$$

The formulas introduced above can be used to determine the dimension, the weight multiplicities, and the character of $V_G(\varpi).$ The problem consisting in determining the composition factors of $V_G(\varpi),$ on the other hand, is essentially equivalent to the determination of weight multiplicities in simple modules for $G$: no closed formula is known to this day, and there seems to be no expectation of finding one in the near future. Altough it is possible to proceed in a recursive fashion, by  arguing on generating sets for weight spaces, those processes are again quite demanding in terms of complexity,  and give no insight on the obtained values. For sufficiently large $p$ and small enough $\varpi,$ other tools are at our disposal, like Kazhdan-Lusztig polynomials \cite{KaLu}, or the Jantzen $p$-sum formula \cite[Proposition 8.19]{Jantzen}. The former allows one to compute weight multiplicities in a recursive fashion, inspired by the study of Verma modules in characteristic zero \cite[Chapter 8]{Humphreys3}. The latter provides a tool for computing all characters of irreducible modules, but generally only in small rank. 
\vspace{5mm}

In this paper, we determine the structure of certain Weyl modules for $G$ in the case where $ \characteristic K \neq 2$ and  $G=\Spin(W),$ with $W$ a $K$-vector space of dimension at least $5.$ In order to do so, we proceed in two steps: inspired by an idea of McNinch \cite{McNinch}, we first determine the composition factors of a well-chosen tilting module for $G,$ in order to reduce the list of possible composition factors for $V_G(\varpi),$ thanks to a generalization of \cite[Proposition 4.6.2]{McNinch}, namely Proposition \ref{chT(omega)_as_a_sum_of_xi's}. Finally, a suitable use of a truncated version of the Jantzen $p$-sum formula (Theorem \ref{How_to_determine_contributions}) yields the desired result. We then  deduce the dimensions of the corresponding irreducible $KG$-modules, and conclude by proving a result on the composition factors of the  restriction to $\SO(W)$ of certain $\SL(W)$-modules.

%%%%%%%%%%%%%%%%%%%%%%%%%%%%%%%%%%%%%%%%%%%%
\subsection{Statements of results}     %%%%%
%%%%%%%%%%%%%%%%%%%%%%%%%%%%%%%%%%%%%%%%%%%%

Assume $\characteristic K\neq 2,$ and let $G$ be a simply connected, simple algebraic group of type $B_n$ $(n\geqslant 2)$ or $D_n$ $(n\geqslant 3).$ Fix a Borel subgroup $B = UT$ of $G,$ where $T$ is a maximal torus of $G$ and $U$ is the unipotent radical of $B,$ let $\Pi = \{\beta_1,\ldots,\beta_n\}$ denote a corresponding base of the root system $\Phi$ of $G,$ and let $\{\varpi_1,\ldots,\varpi_n\}$ be the set of fundamental dominant weights for $T$ corresponding to our choice of base $\Pi,$ ordered as in \cite{Bourbaki}. Also adopt the notation 
$$
\Lambda(B_n)=\{\varpi_i:~1\leqslant i\leqslant n\}\cup \{2\varpi_n\},
$$ 
as well as 
$$
\Lambda(D_n)=\{\varpi_i:~1\leqslant i\leqslant n\}\cup \{2\varpi_{n-1},2\varpi_n,\varpi_{n-1}+\varpi_n\}.
$$

Since we assumed $\characteristic K\neq 2,$ any Weyl module for $G$ having highest weight $\varpi\in \Lambda(G)$ is irreducible (see Lemmas \ref{Dimension_of_various_irreducibles_in_type_B} and \ref{Dimension_of_various_irreducibles_in_type_D}, for instance), and hence the dimension of $V_G(\varpi)$, as well as its weight multiplicities, can be computed using the tools provided by the theory in characteristic zero. In this paper, we thus focus our attention on Weyl modules having slightly more complicated highest weights, namely weights belonging to the set
$$
\Lambda_1(G) = \varpi_1+\Lambda(G).
$$
The result of our investigation (which can be viewed as a generalization of \cite[Lemma 4.9.2]{McNinch}, in which the case $\varpi=\varpi_1+\varpi_2$ is dealt with) is recorded in the following theorem. For $\ell \in \Z_{\geqslant 0}$ a prime, we let  $\epsilon_\ell:\Z_{\geqslant 0} \rightarrow \{0,1\}$ be the map defined by
$$
\epsilon_\ell(z)=
\begin{cases}
1 & \mbox{ if $\ell$ divides $z;$} \cr
0 & \mbox{ otherwise.}
\end{cases}
$$

%-------------------------------%
%-----     Main Result     -----%
%-------------------------------%
\begin{thmintro}\label{Main_Result_1}  
Assume $\characteristic K \neq 2,$ and let $G$ be a simply connected, simple algebraic group of type $B_n$ $,n\geqslant 2$ (resp. $D_n$ $,n\geqslant 3$), over $K.$ Also let  $\varpi$ be as in the first column of  Table \textnormal{\ref{Table_B}} (resp. Table \textnormal{\ref{Table_D}}). Then the structure of the radical $\rad (\varpi)$ of $V_G(\varpi)$ is given by the second column of the table.
\begin{table}[h]
\center
\begin{tabular}{cc}
\hline \hline \\ [-2ex] 
$\varpi$	&	$\rad (\varpi)$	\\  
\hline \hline \\ [-1.5ex]
$2\varpi_1$								&	$K^{\epsilon_p(2n+1)}$	\cr
$\varpi_1+\varpi_j$	$(2\leqslant j\leqslant n-2)$	& 	$L_G(\varpi_{j+1})^{\epsilon_p(j+1)} \oplus L_G(\varpi_{j-1})^{\epsilon_p(2n-j+2)}$							\cr
$\varpi_1+\varpi_{n-1}	$					& 	$L_G(2\varpi_n)^{\epsilon_p(n)} \oplus L_G(\varpi_{n-2})^{\epsilon_p(n+3)}$							\cr 	
$\varpi_1+ \varpi_n$			&  $  L_G( \varpi_n)^{\epsilon_p(2n+1)}$ 	\cr 
$\varpi_1+2\varpi_n	$		& 	$L_G(2\varpi_n)^{\epsilon_p(n+1)} \oplus L_G(\varpi_{n-1})^{\epsilon_p(n+2)}$						  	 							\\[0.15cm]
\hline\cr
\end{tabular}
\caption{Structure of certain Weyl modules for $G$ of type $B_n$ $(n\geqslant 2).$}
\label{Table_B}
\end{table}

\begin{table}[h]
\center
\begin{tabular}{cc}
\hline \hline \\[-2ex] 
$\varpi$	&	$\rad (\varpi)$	\\  
\hline \hline \\ [-1.5ex]
$2\varpi_1$								&	$K^{\epsilon_p(n)}$	\cr
$\varpi_1+\varpi_j$	$(2\leqslant j\leqslant n-3)$	& 	$L_G(\varpi_{j+1})^{\epsilon_p(j+1)} \oplus L_G(\varpi_{j-1})^{\epsilon_p(2n-j+1)}$						\cr
$\varpi_1+\varpi_{n-2}$					&	$L_G(\varpi_{n-1}+\varpi_n)^{\epsilon_p(n-1)} \oplus L_G(\varpi_{n-3})^{\epsilon_p(n+3)}$								\cr
$\varpi_1+\varpi_n	$					& 	$L_G(\varpi_{n-1})^{\epsilon_p(n)}$																	\cr 	$\varpi_1+\varpi_{n-1}+\varpi_n$			&  $L_G(2\varpi_{n-1})^{\epsilon_p(n)} \oplus L_G(2\varpi_{n})^{\epsilon_p(n)}  \oplus L_G(\varpi_{n-2})^{\epsilon_p(n+2)}$  	\cr
$\varpi_1+2\varpi_n	$		& 	$L_G(\varpi_{n-1}+\varpi_n)^{\epsilon_p(n+1)}$						  	 							\\[0.15cm]
\hline\cr
\end{tabular}
\caption{Structure of certain Weyl modules for $G$ of type $D_n$ $(n\geqslant 3).$}
\label{Table_D}
\end{table}
\end{thmintro}
%-------------------------------%
%-------------------------------%

%-----------------------------------------------------%
%----- Remarks on p=2 and on the symplectic case -----%
%-----------------------------------------------------%
\begin{remarksintro}
The assumption on the characteristic of $K$ in Theorem \ref{Main_Result_1} ensures that Weyl modules for $G$ having highest weights belonging to $\Lambda(G)$ (as defined above) are irreducible, thus allowing us to apply the aforementioned generalization of McNinch's result (see Proposition \ref{chT(omega)_as_a_sum_of_xi's}). Also observe that the case where $G$ is of type $C_n$ $(n\geqslant 3)$ is not treated in this paper. The reason is that if $G$ is of type $C_n$ over $K,$ then Weyl modules  having fundamental weights as highest weights are not necessarily irreducible, this even if $p\neq 2.$ (In fact, there is no bound to the possible number of composition factors for such modules, as $n$ grows \cite{Suprunenko}). In particular, the method employed in this paper requiring of Weyl modules having highest weights $\varpi_i,$ $1\leqslant i\leqslant n,$ to be irreducible, would fail to apply in this context. A similar result for $G$ of type $C_n$  would then require a lot more investigation and would probably lead to a much more complicated table. Finally, observe that a table similar to Tables \ref{Table_B} and \ref{Table_D} in the case where $G=D_n$ $(n\geqslant 3)$ and $\varpi\in \varpi_2+\Lambda(D_n)$ can be found in \cite[Theorem 7.3]{Thesis}. However, since it is incomplete, and since the techniques employed  are identical to the ones introduced here, we decided not to include the result in this paper.
\end{remarksintro}
%-----------------------------------------------------%
%-----------------------------------------------------%

As seen above, each irreducible module $KG$-module appearing in the second column of Table \ref{Table_B} or Table \ref{Table_D} of Theorem \ref{Main_Result_1} is isomorphic to its corresponding Weyl module since $p\neq 2.$ In particular, the dimensions of those irreducibles are known, and so one can deduce the dimension of each irreducible $KG$-module having highest weight $\varpi$ as in the first column of the aforementioned tables. We record our findings in the form of a  corollary to Theorem \ref{Main_Result_1}. For simplicity purposes, we let $\delta: \mathbb{Z}\times \Z \to \{0,1\}$ denote the standard Kronecker delta, that is, 
$$
\delta_{a,b}=\begin{cases} 1 & \mbox{if $a=b;$}\cr
0 & \mbox{otherwise}.
\end{cases}
$$

%-------------------------------------------------%     
%-----     Corollary on the irreducibles     -----% 
%-------------------------------------------------%
\begin{corolintro}\label{Introduction: corollary on the dimension (B)}
Assume $\characteristic K \neq 2,$ and let $G$ be a simply connected, simple algebraic group of type $B_n$ $,n\geqslant 2$ (resp. $D_n$ $,n\geqslant 3$), over $K.$ Also let  $\varpi$ be as in the first column of  Table \textnormal{\ref{Table_of_dimensions_main_result_type_B}} (resp. Table  \textnormal{\ref{Table_of_dimensions_main_result_type_D}}). Then the dimension of $L_G(\varpi)$ is given by the second column of the table.

\begin{table}[h]
\center
\begin{tabular}{cc}
\hline \hline \\[-2ex] 
$\varpi$	&	$\dim L_G(\varpi)$	\\  
\hline \hline  
%$\varpi_1+\varpi_j +\delta_{j,n}\varpi_n$ &	$\frac{j(2n-j+1)}{2(n+1)} \begin{pmatrix} 2n+3 \\ j+1 \end{pmatrix} -\epsilon_p(j+1)\begin{pmatrix} 2n+1\\ j+1 \end{pmatrix} -\epsilon_p(2n-j+2)\begin{pmatrix} 2n+1\\j-1 \end{pmatrix}$	 \cr
&   \cr 
$\varpi_1+\varpi_j +\delta_{j,n}\varpi_n$ &	$\begin{pmatrix}
2n+1 \\ j+1 \end{pmatrix} \left( \dfrac{j(2n+3)}{2n-j+2}-\epsilon_p(j+1)-\dfrac{\epsilon_p(2n-j+2)j(j+1)}{(2n-j+2)(2n-j+1)}\right)$	 \cr 
&   \cr
$\varpi_1+\varpi_n$	& 	$2^n(2n-\epsilon_p(2n+1)) $\\
&   \cr 
\hline\cr
\end{tabular}
\caption{Dimension of certain irreducible $KG$-modules  for $G$ of type $B_n$ $(n\geqslant 2).$ Here $1\leqslant j\leqslant n.$}
\label{Table_of_dimensions_main_result_type_B}
\end{table}

\begin{table}[h]
\center
\begin{tabular}{ccc}
\hline \hline \\[-2ex] 
 $\varpi$	&	$\dim L_G(\varpi)$	\\  
\hline \hline \\ [-1.5ex]
%$\varpi_1+\varpi_j + \delta_{j,n-1}\varpi_n$								&	$\frac{j(2n-j)}{2n+1} \begin{pmatrix} 2n+2 \\ j+1 \end{pmatrix}  -\epsilon_p(j+1)\begin{pmatrix} 2n \\ j+1 \end{pmatrix}- \epsilon_p(2n-j+1)\begin{pmatrix} 2n \\ j-1 \end{pmatrix}$						\cr 
& \cr
$\varpi_1+\varpi_j + \delta_{j,n-1}\varpi_n$								&	$\begin{pmatrix}2n \\ j+1\end{pmatrix}	\left( \dfrac{2j(n+1)}{2n-j+1}-\epsilon_p(j+1)-\dfrac{\epsilon_p(2n-j+1)j(j+1)}{(2n-j)(2n-j+1)}\right)$				\cr 
& \cr
 $\varpi_1+\varpi_n$	& 	$2^n(2n-\epsilon_p(2n+1)) $\cr
& \cr
 $\varpi_1+2\varpi_n$	&	$ (n -\epsilon_p(n+1))\begin{pmatrix} 2n \\ n+1 \end{pmatrix}  $\cr & \cr
\hline\cr
\end{tabular}
\caption{Dimension of certain irreducible $KG$-modules  for $G$ of type $D_n$ $(n\geqslant 3).$ Here $1\leqslant j\leqslant n-1.$}
\label{Table_of_dimensions_main_result_type_D}
\end{table}
\end{corolintro} 
%-------------------------------------------------%
%-------------------------------------------------% 

Let $W$ be a finite-dimensional $K$-space of dimension at least $5,$ and let $Y=\SL(W),$ that is, $Y$ is a simply connected, simple algebraic group of type $A_{\dim W-1}$ over $K.$ Fix a Borel subgroup $B_Y$ of $Y,$ containing a maximal torus $T_Y$ of $Y,$ and let $\{\lambda_1,\ldots,\lambda_{\dim W-1}\}$ denote the corresponding fundamental weights, ordered as in \cite{Bourbaki}. Also consider a  maximal, closed, connected subgroup $G=\SO(W)$ of $Y.$ Then $G$ is of type $B_n$ ($n\geqslant 2$) if $\dim W =2n+1,$ and of type $D_n$ ($n\geqslant 3$) if $\dim W=2n.$  Without loss of generality, we suppose that $T,$ $B,$ and hence $\{\varpi_1,\ldots,\varpi_n\},$ are chosen in such a way that $\lambda_i|_T = \varpi_i$ for  $1\leqslant i\leqslant n-2,$ $\lambda_{n-1}|_T=\varpi_{n-1}+\epsilon_2(\dim W)\varpi_n,$  and $\lambda_n|_T = 2\varpi_n.$ 
\vspace{5mm}

If  $\characteristic K \neq 2 $ and if $V$ is an irreducible $KY$-module having highest weight $\lambda_i,$ $1 \leqslant i\leqslant \dim W -1,$ then the restriction of $V$ to $G$ is irreducible as well by \cite[Theorem 1, Table 1 (I$_2,$ I$_3,$ I$_4,$ I$_5$)]{Se}. We thus conclude this paper by giving a description of the composition factors of the restriction to $G$ of irreducible $KY$-modules having slightly more complicated highest weights, namely weights of the form $\lambda=\lambda_1+\lambda_j,$ where $1\leqslant j\leqslant \dim W-1.$

\begin{propintro}\label{Restrictions_of_irreducibles_of_type_A}
Assume $\characteristic K \neq 2,$ and let $Y=\SL(W)$ and $G=\SO(W)$ be as above. Also let  $\lambda$ be as in the first column of  Table \textnormal{\ref{Restriction_of_V_Y(lambda)_to_G_of_type_B}} (resp. Table   \textnormal{\ref{Restriction_of_V_Y(lambda)_to_G_of_type_D}}). Then the composition factors of the restriction of $L_Y(\lambda)$ to $G$ is given by the second column of the table. 

\begin{table}[h]
\center
\begin{tabular}{ccc}
\hline \hline \\[-2ex] 
$\lambda$	&	$L_Y(\lambda)|_G$ \\
\hline \hline \\ [-1.5ex]
$\lambda_1+\lambda_j~(1\leqslant j\leqslant n)$ 			& $\varpi~\big/~ \varpi_{j-1} ~\big/~ \varpi_{j-1}^{\epsilon_p(2n-j+2)}$ \cr
$\lambda_1+\lambda_{n+1}$					& $\varpi ~\big/~ 2\varpi_n ~\big/~ 2\varpi_{n}^{ \epsilon_p(n+1)}$			\cr
$\lambda_1+\lambda_{n+2}$					& $\varpi~\big/~ 2\varpi_n ~\big/~ 2\varpi_{n}^{\epsilon_p(n)} $			\cr
$\lambda_1+\lambda_j~(n+3\leqslant j\leqslant 2n)$		& $\varpi~\big/~ \varpi_{2n-j+2}  ~\big/~ \varpi_{2n-j+2}^{ \epsilon_p(2n-j+2)}  $	\cr 
\hline\cr
\end{tabular}
\caption{$KG$-composition factors of certain irreducible $KY$-modules for $Y=\SL(W),$ $G=\SO(W),$ $\dim W=2n+1$ for some $n\geqslant 2.$}
\label{Restriction_of_V_Y(lambda)_to_G_of_type_B}
\end{table} 

\begin{table}[h]
\center
\begin{tabular}{ccc}
\hline \hline \\[-2ex] 
$\lambda$	&	$L_Y(\lambda)|_G$	\\
\hline \hline \\ [-1.5ex]
$\lambda_1+\lambda_j~(1\leqslant j\leqslant n-1)$	& $\varpi ~\big/~  \varpi_{j-1} ~\big/~ \varpi_{j-1}^{\epsilon_p(2n-j+1)} $		\cr
$\lambda_1+\lambda_n$	& 	$\varpi ~\big/~ \varpi_1+2\varpi_{n-1} ~\big/~ \varpi_1 + \varpi_n ~\big/~  (\varpi_1 + \varpi_n)^{\epsilon_p(n+1)} $ \cr
$\lambda_1+\lambda_{n+1}$					& $
\varpi ~\big/~ 2\varpi_{n-1}^{1+\epsilon_p(n)} ~\big/~  2\varpi_{n}^{1+\epsilon_p(n) }	~\big/~ $			\cr
$\lambda_1+\lambda_{n+2}$					& $\varpi ~\big/~ \varpi_{n-1} +\varpi_n	 ~\big/~ (\varpi_{n-1} +\varpi_n)^{\epsilon_p(n-1)} $			\cr
$\lambda_1+\lambda_j~(n+3\leqslant j\leqslant 2n-1)$		& $\varpi ~\big/~ \varpi_{2n-j+1}   ~\big/~ \varpi_{2n-j+1} ^{\epsilon_p(2n-j+1)} $	\cr 
\hline\cr
\end{tabular}
\caption{$KG$-composition factors of certain irreducible $KY$-modules for $Y=\SL(W),$ $G=\SO(W),$ $\dim W=2n$ for some $n\geqslant 3.$}
\label{Restriction_of_V_Y(lambda)_to_G_of_type_D}
\end{table} 
\end{propintro}

%%%%%%%%%%%%%%%%%%%%%%%%%%%%%%%%%%%%%%%%%%%%%%%%%%%%%%%%%%%%%%%%
%----- Preliminaries     -----%%%%%%%%%%%%%%%%%%%%%%%%%%%%%%%%%%      
%%%%%%%%%%%%%%%%%%%%%%%%%%%%%%%%%%%%%%%%%%%%%%%%%%%%%%%%%%%%%%%%
\section{Preliminaries}\label{[Section] Preliminaries}     %%%%% %%%%%%%%%%%%%%%%%%%%%%%%%%%%%%%%%%%%%%%%%%%%%%%%%%%%%%%%%%%%%%%%
%%%%%%%%%%%%%%%%%%%%%%%%%%%%%%%%%%%%%%%%%%%%%%%%%%%%%%%%%%%%%%%%

Let $K$ be an algebraically closed field having characteristic $p\geqslant 0.$ In this section, we recall some elementary properties about representations of simple algebraic groups over $K.$  Unless specified otherwise, most of the results presented here can be found in \cite{Bourbaki}, \cite{Humphreys2}, or \cite{Humphreys1}.

%%%%%%%%%%%%%%%%%%%%%%%%%%%%%%%%%%%%%%%%%%%%%%%
\subsection{Notation}\label{Notation}     %%%%% %%%%%%%%%%%%%%%%%%%%%%%%%%%%%%%%%%%%%%%%%%%%%%% 

We first fix some notation that will be used for the rest of the paper. Let $G$ be a simply connected, simple algebraic group over $K.$ Also fix a Borel subgroup $B=UT$ of $G,$ where $T$ is a maximal torus of $G$ and $U$ denotes the unipotent radical of $B.$ Let $n=\rank G = \dim T$ and let $\Pi=\{\alpha_1,\ldots,\alpha_n\}$ be a corresponding base of the root system $\Phi=\Phi^+\sqcup \Phi^-$ of $G,$ where $\Phi^+$ and $\Phi^-$ denote the sets of positive and negative roots of $G,$ respectively. Let
$$
X(T)=\Hom (T,K^*) 
$$
denote the character group of $T,$ and set $X(T)_{\mathbb{R}} = X(T) \otimes_{\Z} \R.$ Also, for $\alpha\in \Phi,$ define the reflection $s_{\alpha}:X(T)_{\R} \to X(T)_{\R}$ relative to $\alpha$ by $s_{\alpha}(\lambda)=\lambda-\langle \lambda,\alpha\rangle\alpha,$ where $\langle \lambda, \alpha\rangle = 2(\lambda,\alpha) (\alpha,\alpha)^{-1}$ for   $\lambda,\alpha$ with $\alpha \neq 0,$ and $(-,-)$ denotes the usual  inner product on $X(T)_\R.$ Denote  by $\mathscr{W}$ the finite group $\langle s_{\alpha_r}:1\leqslant r\leqslant n\rangle,$ called the \emph{Weyl group} of $G.$ Recall the existence of a partial ordering on $X(T)_{\R},$ defined by $\mu\preccurlyeq \lambda$ if and only if $\lambda-\mu \in \Gamma,$ where $\Gamma$ denotes the monoid of $\mathbb{Z}_{\geqslant 0}$-linear combinations of simple roots. (We also  write $\mu\prec \lambda$ to indicate that $\mu\preccurlyeq \lambda$ and  $\mu \neq \lambda.$) In addition, let $\{\lambda_1,\ldots,\lambda_n\}$ be the set of fundamental weights for $T$ corresponding to our choice of base $\Pi,$ that is $\langle \lambda_i,\alpha_j\rangle =\delta_{ij}$ for every $1\leqslant i,j\leqslant n.$ Set
$$
X^+(T)=\{\lambda\in X(T): \langle \lambda,\alpha_r\rangle \geqslant 0 \mbox{ for every } 1\leqslant r\leqslant n\}
$$
and call a character $\lambda \in X^+(T)$ a \emph{dominant character}. Every such character can be written as a $\Z_{\geqslant 0}$-linear combination  $\lambda=\sum_{r=1}^n{a_r\lambda_r},$ where $a_1,\ldots,a_n \in \Z_{\geqslant 0}.$  

%Finally, we use the notation $$U_{\alpha}=\{x_{\alpha}(c):c\in K\}$$ to denote the root subgroup of $G$ corresponding to the root $\alpha \in \Phi,$ (i.e. $x_{\alpha}:K\to G$ is an  isomorphism of algebraic groups such that $tx_{\alpha}(c)t^{-1}=x_{\alpha}(\alpha(t)c)$ for   $t\in T$ and $c\in K$) and we write $\Lie(G)$ for the Lie algebra of $G.$ 

%%%%%%%%%%%%%%%%%%%%%%%%%%%%%%%%%%%%%%%%%%%%%%%%%%
\subsection{Rational representations}     	%%%%%
\label{Rational representations}     		%%%%% %%%%%%%%%%%%%%%%%%%%%%%%%%%%%%%%%%%%%%%%%%%%%%%%%%

In this section, we recall some elementary properties of rational modules for  semisimple algebraic groups, starting by investigating weights and their multiplicities. Unless specified otherwise, the results recorded here can be found in \cite[Chapter XI, Section 31]{Humphreys2}. Let $V$ denote a finite-dimensional, rational $KG$-module. Then  $V$ can be decomposed into a direct sum of $KT$-modules 
$$
V= \bigoplus_{\mu\in X(T)} V_\mu,
$$
where $V_\mu=\{v\in V: t\cdot v=\mu(t) v \mbox{ for all $t\in T$}\}$  for   $\mu\in X(T).$ A character $\mu \in X(T)$ with $V_\mu\neq 0$ is called a \emph{$T$-weight} of $V,$ and $V_\mu$ is said to be its corresponding \emph{weight space}. The dimension of $V_\mu$ is called the \emph{multiplicity of $\mu$ in $V$} and is denoted by $\m_V(\mu).$ Write $\Lambda(V)$ to denote the set of $T$-weights of $V,$ and  set $\Lambda^+(V)=\Lambda(V)\cap X^+(T).$ Any weight in $\Lambda^+(V)$ is called \emph{dominant}.
\vspace{5mm}

The natural action of the Weyl group $\mathscr{W}$ of $G$ on $X(T)$ induces an action on $\Lambda(V)$ and we say that $\lambda,\mu\in X(T)$ are    $\mathscr{W}$\emph{-conjugate}  if there exists $w\in \mathscr{W}$ such that $w \lambda=\mu.$ It is well-known (see \cite[Section 13.2, Lemma A]{Humphreys1}, for example) that $X^+(T)$ is a fundamental domain for the latter action, that is, each weight in $X(T)$ is $\mathscr{W}$-conjugate to a unique dominant weight. Also, if $\lambda\in X^+(T),$ then $w\lambda \preccurlyeq \lambda$ for every $w\in \mathscr{W}.$ Finally, $\Lambda(V)$ is a union of $\mathscr{W}$-orbits and all weights in a $\mathscr{W}$-orbit have the same multiplicity.
\vspace{5mm}

Now by the Lie-Kolchin Theorem (\cite[Theorem 17.6]{Humphreys2}), there exists $0\neq v^+\in V$ such that $\langle v^+ \rangle_K$ is invariant under the action of $B.$ We call such a vector $v^+$ a \emph{maximal vector in $V$ for $B$}. Note that since $\langle v^+ \rangle_K$ is stabilized by any maximal torus of $B,$ there exists $\lambda \in X(T)$ such that $v^+\in V_{\lambda}.$ In fact, one can show that $\lambda \in X^+(T).$ It is well-known that isomorphism classes of finite-dimensional, irreducible, rational modules are in one-to-one correspondence with dominant weights for $T.$ In this paper, we shall write $L_G(\lambda)$ for the irreducible $KG$-module having highest weight $\lambda,$ obtained as a quotient of the corresponding Weyl module $V_G(\lambda)$ by its unique maximal submodule $\rad (\lambda),$ that is, 
$$
L_G(\lambda)=\bigquotient{V_G(\lambda)}{\rad (\lambda)}.
$$

Clearly each $\lambda\in X(T)$ determines a $1$-dimensional $KT$-module $K_{\lambda}$ on which every $t\in T$ acts as multiplication by $\lambda(t)$ and one observes that we get a $KB$-module structure on $K_{\lambda},$ given by $(ut)x=\lambda(t)x,$ for every $ut\in B$ and $x\in K_\lambda.$  For $r\geqslant 0,$ we let $H^r(-)=H^r(G/B,-)$ denote the $r^{th}$ derived functor  of the left exact functor $\ind _B^G(-) $ and write $H^r(\lambda)=H^r( K_{\lambda}).$ It turns out (see \cite[II, 2.13]{Jantzen}) that if $\lambda \in X^+(T),$ then $H^0(\lambda)\cong V_G(-w_0 \lambda)^*,$ where $w_0$ denotes the longest element in the Weyl group of $G.$ Consequently $L_G(\lambda)\cong L_G(-w_0 \lambda)^*$ is the unique irreducible submodule of $H^0(\lambda) $ and hence is the \emph{socle} of $H^0(\lambda),$ written $\soc (\lambda).$ We refer the reader to \cite[Section 2.1]{Jantzen} for more details.  Finally, the following result makes it easier to compute weight multiplicities in certain situations.

%------------------------------------------------------------------------------------\
%----- Restriction to a Levi subgroup in order to compute weight multiplicities -----\
%------------------------------------------------------------------------------------\
\begin{lem}\label{[Preliminaries] Parabolic embeddings: restriction to Levi subgroup}
Let $V=L_G(\lambda)$ be an irreducible $KG$-module having highest weight $\lambda\in X^+(T).$ Let $J\subset \Pi$ and $\mu\in \Lambda ^+(V)$ be such that $\mu=\lambda-\sum_{\alpha \in J}{c_{\alpha} \alpha}.$ Also write $H=\langle U_{\pm \alpha} : \alpha \in J \rangle .$ Then  $\m_V(\mu)=\m_{V'}(\mu'),$ where $\mu' =\mu|_{T_H},$ $V'=L_H (\lambda|_{T_H}).$
\end{lem}

\begin{proof}
Let $P$ be the standard parabolic subgroup of $G$ corresponding to the subset $J,$ so that $H$ is the derived subgroup of a Levi factor of $P.$ The weight space $V_{\mu}$ lies in the fixed point space of the unipotent radical of $P,$ which is isomorphic to $L_H(\lambda|_{T_H})$ by \cite[Proposition 2.11]{Jantzen}. The result then follows.
\end{proof}
%------------------------------------------------------------------------------------|
%------------------------------------------------------------------------------------|

%%%%%%%%%%%%%%%%%%%%%%%%%%%%%%%%%%%%%%%%%%%%%%%%%% 
\subsection{Some dimension calculations}     %%%%%
%%%%%%%%%%%%%%%%%%%%%%%%%%%%%%%%%%%%%%%%%%%%%%%%%% 

In this section, $G$ denotes a simply connected, simple algebraic group of rank $n$ over $K$ and $V=L_G(\lambda)$ an irreducible $KG$-module having $p$-restricted highest weight $\lambda\in X^+(T).$ In general, the dimension of $V$ is unknown, or at least there is no known formula holding for $\lambda$ arbitrary. Nevertheless, the dimension of $V_G(\lambda)$ is given by the well-known \emph{Weyl's dimension formula} (see \cite[Section 24.3]{Humphreys1}, for instance). The following result consists in a slightly modified version of the latter formula, which allows one to compute the dimension of a given Weyl module recursively. The proof, being straightforward, is omitted here.

%-----------------------------------\
%------ Weyl's Degree Formula ------\
%-----------------------------------\
\begin{thm}[Weyl's degree formula]\label{Weyl's formula}  
Set $\Phi^+_{1}=\left\{\alpha=\sum_{r=1}^n{a_r\alpha_r} \in \Phi^+:a_1>0\right\} $ and let $L$ denote a Levi subgroup of $G$ corresponding to the simple roots $\alpha_2,\ldots,\alpha_n.$  Then 
$$
\dim V_G(\lambda)	= \left(\prod_{\alpha \in \Phi^+_{1}}{\frac{\langle \lambda+\rho,\alpha\rangle}{\langle \rho,\alpha\rangle}}\right)\dim V_{L}(\lambda|_{T\cap L}).
$$
\end{thm}
%-----------------------------------|
%-----------------------------------|

We now record some information on the dimension of various irreducible $KG$-modules for $G$ of type $A_n$ $(n\geqslant 1),$ $B_n$ ($n\geqslant 2$), and $D_n$ ($n\geqslant 3$) over $K,$ starting by dealing with the former case. We say that a dominant $T$-weight  $\lambda$ is \emph{$p$-restricted} if either $p=0,$ or  if $0 \leqslant \langle \lambda, \alpha \rangle <p $ for $\alpha\in \Pi.$ Also, for $\ell \in \Z_{\geqslant 0}$ a prime, we let $\epsilon_\ell:\Z_{\geqslant 0} \rightarrow \{0,1\}$ be the map defined by
$$
\epsilon_\ell(z)=
\begin{cases}
1 & \mbox{ if $\ell$ divides $z;$} \cr
0 & \mbox{ otherwise.}
\end{cases}
$$

%---------------------------------------------------------------\
%------ Dimension of various irreducibles for G of type A ------\
%---------------------------------------------------------------\
\begin{lem}\label{Dimension_of_various_irreducibles_in_type_A}
Let $G$ be a simple algebraic group of type $A_n$ $(n\geqslant 2)$ over $K $ and consider an irreducible $KG$-module $V=L_G(\lambda)$ having $p$-restricted highest weight $\lambda \in X^+(T).$ Then the following assertions hold.
\begin{enumerate}
\item \label{Various_dimensions_in_type_A_1} If $\lambda=a\lambda_1$ for some $a \in \Z_{\geqslant 1},$ then $V =V_G(\lambda) \cong \SymP^a W,$ where $\SymP^a W$ denotes the $a^{th}$ symmetric power of the natural $KG$-module $W.$ 
\item \label{Various_dimensions_in_type_A_2} If $\lambda=\lambda_i$ for some $1\leqslant i\leqslant n,$ then $V=V_G(\lambda)\cong \Lambda ^i W,$ where $\Lambda^i W$ denotes the $i^{th}$ exterior power of the natural $KG$-module $W.$ 
\item \label{Various_dimensions_in_type_A_3} If $\lambda= \lambda_1+\lambda_j$ for some $2\leqslant j\leqslant n,$ then $V=V_G(\lambda)$ if and only if $p\nmid j+1.$ 
\end{enumerate}
Furthermore, if $\lambda$ is as in \textnormal{\ref{Various_dimensions_in_type_A_1}}, \textnormal{\ref{Various_dimensions_in_type_A_2}}, or \textnormal{\ref{Various_dimensions_in_type_A_3}} above, then the dimension of $V$ is given by the second column of Table \textnormal{\ref{Various_dimensions_in_type_A}}.
\begin{table}[h]
\center
\begin{tabular}{ccc}
\hline \hline \\[-2ex] 
 			 	$\lambda$		&	$\dim L_G(\lambda)$ \\  
\hline \hline \\ [-1.5ex]
 	$a\lambda_1$	 $(a\geqslant 1)$	& $\begin{pmatrix} a+n \\ a \end{pmatrix}$	\\ 
 	& & \cr
 	$\lambda_i $ $(1\leqslant i\leqslant n)$		& $\begin{pmatrix} n+1 \\ i \end{pmatrix}$	\cr
 	 	& & \cr
 	$\lambda_1+\lambda_j$ $(2\leqslant j\leqslant n)$		& $ j \begin{pmatrix} n+2 \\ j+1 \end{pmatrix} -\epsilon_p(j+1) \begin{pmatrix} n+1 \\ j+1 \end{pmatrix} $ \\[0.4cm]
\hline\cr
\end{tabular}
\caption{Dimension of certain irreducible modules for $G$ of type $A_n $ $(n\geqslant 2).$}
\label{Various_dimensions_in_type_A}
\end{table}
\end{lem}

\begin{proof}
We refer the reader to \cite[Lemma 1.14]{Se} for a proof of \ref{Various_dimensions_in_type_A_1}, and assume $\lambda$ is as in \ref{Various_dimensions_in_type_A_2}, in which case  $\Lambda^+(\lambda)=\{\lambda\}.$ Since $\m_V(\lambda)=\m_{V_G(\lambda)}(\lambda)=1,$ the weight $\lambda$ cannot afford the highest weight of a second composition factor of $V_G(\lambda)$ by \cite{Pr}.  Consequently  $V=V_G(\lambda)$ as desired, and an application of Theorem \ref{Weyl's formula}   yields the assertion on the dimension of $V.$  Now fix a  $K$-basis $\{v_1,\ldots,v_{n+1}\}$ for $W,$ where $v_1\in W_{\lambda_1},$ $v_{r+1}\in W_{\lambda_1-(\alpha_1+\cdots+\alpha_r)}$ for $1\leqslant r\leqslant n.$  Then   
$$
\Lambda^i W =\langle v_{r_1} \wedge \ldots \wedge v_{r_i}:1\leqslant r_1<r_2<\ldots <r_i\leqslant n + 1 \rangle_K
$$ 
by definition, and one easily checks that  $v_1\wedge v_2\wedge \ldots \wedge  v_i$ is a maximal vector of weight $\lambda $  in $\Lambda^i W.$  Hence $\Lambda^i W$ admits a composition factor isomorphic to $V.$ An application of Theorem \ref{Weyl's formula} then yields $ \dim V = \dim \Lambda^iW,$ thus showing that the second assertion holds as well. 
 
Finally, let $\lambda$ be as in \ref{Various_dimensions_in_type_A_3}, and observe that $\Lambda^+(\lambda)=\{\lambda,\lambda_{j+1}\},$ where we adopt the notation $\lambda_{n+1}=0.$ As above, applying \cite{Pr} shows that  $V_G(\lambda)$ is  reducible if and only if $\lambda_{j+1}$ affords the highest weight of a composition factor of $V_G(\lambda).$ An application of \cite[Proposition 8.6]{Se} then shows that the latter assertion holds if and only if $p$ divides $j+1,$ in which case $\m_V(\lambda_{j+1})=\m_{V_G(\lambda) }(\lambda_{j+1})-1.$ Consequently $\dim V=\dim V_G(\lambda)-\epsilon_p(j+1)\dim L_G(\lambda_{j+1}),$ and Theorem \ref{Weyl's formula} together with our knowledge of the dimension of exterior powers allow us to conclude.
\end{proof}
%------------------------------------------------------|
%------------------------------------------------------|

We next prove a result similar to Lemma \ref{Dimension_of_various_irreducibles_in_type_A}, for certain irreducible $KG$-modules in the case where $G$ is of type $B_n$ $(n\geqslant 2)$ and $\characteristic K \neq 2.$

%------------------------------------------------------\
%------ Dimension of L(lambda_i) for G of type B ------\
%------------------------------------------------------\
\begin{lem}\label{Dimension_of_various_irreducibles_in_type_B}
Assume $p\neq 2,$ let $G$ be a simple algebraic group of type $B_n$ $(n\geqslant 2)$ over $K,$  and consider an irreducible $KG$-module $V=L_G(\lambda)$ having highest weight $\lambda \in \{ \lambda_i+\delta_{i,n} \lambda_n : 1\leqslant i \leqslant  n\}\cup\{\lambda_n\}.$ Then  $V=V_G(\lambda)$ and the dimension of $V$ is given by the second column of Table \textnormal{\ref{Various_dimensions_in_type_B}}.
\begin{table}[h]
\center
\begin{tabular}{ccc}
\hline \hline \\[-2ex] 
 			 	$\lambda$		&	$\dim L_G(\lambda)$ \\  
\hline \hline \\ [-0.5ex]
$\lambda_i + \delta_{i,n}\lambda_n$ $(1\leqslant i\leqslant n)$		& $\begin{pmatrix} 2n+1 \\ i \end{pmatrix}$	 \cr 
& & \cr 
$ \lambda_n$		& $2^n$ \\[0.5ex]
\hline\cr
\end{tabular}
\caption{Dimension of certain irreducible modules for $G$ of type $B_n$ $(n\geqslant 2).$}
\label{Various_dimensions_in_type_B}
\end{table} 
\end{lem}

\begin{proof}
First consider a dominant $T$-weight  $\lambda \in \{\lambda_i+\delta_{i,n}\lambda_n: 1\leqslant i \leqslant  n\},$ and embed $G$ in a simply connected, simple algebraic group $Y$ of type $A_{2n}$ over $K$ in the usual way. (Observe that this forces $G=\SO_{2n+1}(K),$ that is, $G$ is not simply connected. However, the proof does not rely on $G$ being simply connected and so the argument remains valid.) By \cite[Theorem 1, Table 1 (I$_2,$ I$_3$)]{Se}, the irreducible module $L_G(\lambda )$ is isomorphic to the restriction to $G$ of a suitable exterior power of the natural module for $Y.$ Using this observation together with Lemma \ref{Dimension_of_various_irreducibles_in_type_A}, one  deduces the desired assertions on $V$ in the situation where $\lambda$ is as in the first row of the table. Finally, in the case where $\lambda=\lambda_n,$ we get that $\Lambda^+(\lambda)=\{\lambda\}$ and hence $V_G(\lambda)$ is irreducible by \cite{Pr}. The assertion on the dimension of $V$ then immediately follows from Theorem \ref{Weyl's formula}. 
\end{proof}
%------------------------------------------------------|
%------------------------------------------------------|

%-------------------------%
%----- Remark on p=2 -----%
%-------------------------%
\begin{remark}
The structure of a Weyl module  $V_G(\lambda)$ with highest weight $\lambda$ as in the statement of Lemma \ref{Dimension_of_various_irreducibles_in_type_B} is more complex in the situation where $\characteristic K =2$ (see \cite{MR2964276}, for instance). In particular $V_G(\lambda)$ is in general not irreducible and hence not tilting (see Definition \ref{definition of tilting module}).
\end{remark}
%-------------------------%
%-------------------------%

We next prove a result similar to Lemmas \ref{Dimension_of_various_irreducibles_in_type_A} and \ref{Dimension_of_various_irreducibles_in_type_B} for certain irreducible $KG$-modules in the case where $G$ is of type $D_n$ $(n\geqslant 3)$ and $\characteristic K \neq 2.$

%--------------------------------------------------------------%
%------ Dimension of certain irreducibles for G of type D -----%
%--------------------------------------------------------------%
\begin{lem}\label{Dimension_of_various_irreducibles_in_type_D}
Assume $p\neq 2,$ let $G$ be a simple algebraic group of type $D_n$ $(n\geqslant 3)$ over $K,$  and consider an irreducible $KG$-module $V=L_G(\lambda)$ having  highest weight $\lambda \in \{\lambda_i+\delta_{i,n-1}\lambda_n: 1\leqslant i < n\}\cup \{2\lambda_{n-1} \}.$ Then  $V=V_G(\lambda)$ and the dimension of $V$ is given by the second column of Table \textnormal{\ref{Various_dimensions_in_type_D}}.
\begin{table}[h]
\center
\begin{tabular}{ccc}
\hline \hline \\[-2ex] 
 			 	$\lambda$		&	$\dim L_G(\lambda)$ \\  
\hline \hline \\ [-1.5ex]
 	$\lambda_i+\delta_{i,n-1}\lambda_n$		& $\begin{pmatrix} 2n \\ i \end{pmatrix}$	 \cr
 	& & \cr
					 	$\lambda_{n}$		& $2^{n-1}$	\cr 
					 	& & \cr
					 	$2\lambda_{n-1} $		& $\frac{1}{2}\begin{pmatrix} 2n \\ n \end{pmatrix}$ 
					 	\\[0.4cm]
\hline\cr
\end{tabular}
\caption{Dimension of certain irreducible modules for $G$ of type $D_n$ $(n\geqslant 3).$}
\label{Various_dimensions_in_type_D}
\end{table} 
\end{lem}

\begin{proof} 
First consider a dominant $T$-weight  $\lambda \in \{\lambda_i+\delta_{i,n-1}\lambda_n: 1\leqslant i < n\},$ and embed $G$ in a simply connected, simple algebraic group $Y$ of type $A_{2n-1}$ over $K,$ as in the proof of Lemma \ref{Dimension_of_various_irreducibles_in_type_B}. (Again, this yields $G=\SO_{2n}(K).)$ By \cite[Theorem 1, Table 1 (I$_4,$ I$_5$)]{Se}, the irreducible module $L_G( \lambda)$ is isomorphic to the restriction to $G$ of a suitable exterior power of the natural module for $Y.$ Using this observation together with Lemma \ref{Dimension_of_various_irreducibles_in_type_A}, one  checks that the assertions on $V$ hold in this situation. Next assume $\lambda= \lambda_{n},$ in which case $\Lambda^+(\lambda)=\{\lambda\}$ and so $V_G(\lambda)$ is irreducible by \cite{Pr}. Finally, we refer the reader to \cite[Lemma 2.3.6]{BGT} for a proof of the assertions in the situation where $\lambda=2\lambda_{n-1}.$
\end{proof}
%------------------------------------------------------%
%------------------------------------------------------%

%%%%%%%%%%%%%%%%%%%%%%%%%%%%%%%%%%%%%%%%%%%%%%%%%%%%%%%%%%%%%%%%%%%%%%%%%%%%%%
\subsection{Formal character and dot action}\label{Formal_character}     %%%%%      
%%%%%%%%%%%%%%%%%%%%%%%%%%%%%%%%%%%%%%%%%%%%%%%%%%%%%%%%%%%%%%%%%%%%%%%%%%%%%%

Let $\{e^{\mu}\}_{\mu\in X(T)}$ denote the standard basis of the group ring $\Z[X(T)]$ over $\Z.$ The Weyl group $\mathscr{W}$ of $G$ acts on $\Z[X(T)]$ by $we^{\mu}=e^{w\mu},$ $w\in \mathscr{W},$ $\mu\in X(T),$ and we write $\Z[X(T)]^{\mathscr{W}}$ to denote the set of fixed points. The \emph{formal character} of a given $KG$-module  $V$ is defined by
$$
\ch V = \sum_{\mu\in X(T)}{\m_V(\mu) e^{\mu}} \in \Z[X(T)]^{\mathscr{W}}.
$$

Formal characters are valuable tools to study finite-dimensional, rational modules. Following the ideas in \cite[Section II.5.5]{Jantzen}, we also associate to every $T$-weight $\lambda\in X(T)$ the linear polynomial
$$
\chi(\lambda)=\sum_{r\geqslant 0}(-1)^r\ch H^r(\lambda).
$$
If $\lambda\in X^+(T),$ Kempf's vanishing Theorem \cite[II, 4.5]{Jantzen} shows that $H^r(\lambda)=0$ for $r>0$ and hence $\chi(\lambda)=\ch H^0(\lambda)$ in this case. In addition, recall from \cite[II, 2.13]{Jantzen} that if $\lambda\in X^+(T),$ then $\chi(\lambda)=\ch V_G(\lambda)$ as well. One  shows  (see \cite[II, 5.8]{Jantzen}) that each of $\{\chi(\lambda)\}_{\lambda\in X^+(T)}$ and $\{\ch L_G(\lambda)\}_{\lambda\in X^+(T)}$ forms a $\Z$-basis of $\Z[X(T)]^{\mathscr{W}}.$ 

For a $T$-weight $\mu \in X(T)$ with $\mu \prec \lambda,$  we also introduce a ``truncated'' version of $\chi(\lambda),$ which shall prove useful later on in the paper: 
$$
\chi_\mu(\lambda)= \ch L_G(\lambda)+ \sum_{\substack{ \eta \in X^+(T) \\ \mu \preccurlyeq \eta \prec\lambda}}{[V_G(\lambda),L_G(\eta)]\ch L_G(\eta)}.
$$

Let $\rho$ denote the half-sum of all positive roots in $\Phi,$ or equivalently, the sum of all fundamental weights. The \emph{dot action} of $\mathscr{W}$ on $X(T)$ is given by the formula $w\cdot \lambda=w(\lambda+\rho)-\rho,$ for  $w\in \mathscr{W}$ and $\lambda\in X(T).$
One easily sees that 
$$
\mathscr{D}=\{\lambda\in X(T):\langle \lambda+\rho,\alpha \rangle \geqslant 0 \mbox{ for every }\alpha\in \Phi^+\}
$$
is a fundamental domain for the dot action of $\mathscr{W}$ on $X(T)$:  for every $\mu \in X(T),$ there exist  $w\in \mathscr{W}$ and a unique $\lambda\in \mathscr{D}$ such that $w\cdot \mu =\lambda.$ This observation, together with the next result, provide the necessary tools to compute $\chi(\lambda)$ for any given $\lambda \in X(T),$ without having to consider higher homology groups.  For $w\in \mathscr{W},$ we write $\det (w)$ for the determinant of $w$ as an invertible linear transformation of $X(T)_{\R}.$

%---------------------------------------------------------------\
%------ Tools to compute chi(lambda) for a general lambda ------\
%---------------------------------------------------------------\
\begin{lem}\label{[Preliminaries] The Jantzen p-sum formula: Tools to compute chi(lambda)}
Let $\lambda\in X(T)$ and $w\in \mathscr{W}.$ Then $\chi(w\cdot \lambda)=\det(w)\chi(\lambda).$ Moreover, if $\lambda \in \mathscr{D}$ is not in $X^+(T),$ then $\chi(\lambda)=0.$ 
\end{lem}

\begin{proof}
The first assertion immediately follows from \cite[II, 5.9 (1)]{Jantzen} and we refer the reader to \cite[II, 5.5]{Jantzen} for a proof of the second.
\end{proof}
%---------------------------------------------------------------|
%---------------------------------------------------------------|

%%%%%%%%%%%%%%%%%%%%%%%%%%%%%%%%%%%%%%%%%%%%%%%%%%%%%%%%%%%%%%%%%%%%%%%%%%%%%%%%%%%%%%%%%%%%%%%%%%%%%%%%%%%%%%%%%%%%%%%%%%%%%%%%%%%%%%%%%%%%%%%%%%%%%%%%% 
\subsection{Filtrations and extensions of modules}     %-------------------------------------------------------------------------------------------------
%%%%%%%%%%%%%%%%%%%%%%%%%%%%%%%%%%%%%%%%%%%%%%%%%%%%%%%%%%%%%%%%%%%%%%%%%%%%%%%%%%%%%%%%%%%%%%%%%%%%%%%%%%%%%%%%%%%%%%%%%%%%%%%%%%%%%%%%%%%%%%%%%%%%%%%%% 

In this section, we introduce some notation and recall a few basic results concerning filtrations and extensions of $KG$-modules. For such a module $V$ and for $\mu\in X^+(T)$ a dominant weight, we write $[V,L_G(\mu)]$ to denote the number of times the irreducible $L_G(\mu)$ occurs as a composition factor of $V.$ Also, we adopt the notation $V=\mu_1^{m_1}/\mu_2^{m_2}/\ldots/\mu_s^{m_s}$ to indicate that $V$ is a $KG$-module with same composition factors as  $ L_G(\mu_1)^{m_1}\oplus \cdots \oplus L_G(\mu_s)^{m_s},$ where $m_1,\ldots,m_s\in \mathbb{Z}_{>0}.$

%------------------------------------------------\
%------ MODULES WITH WEYL/GOOD FILTRATIONS ------\
%------------------------------------------------\
\begin{defin}\label{definition of tilting module}
A filtration $V=V^0 \supseteq V^1 \supseteq \ldots \supseteq V^{r}\supseteq V^{r+1}=0$ of $V$ is called a \emph{Weyl filtration} if for every $0\leqslant i\leqslant r,$ there exists a weight $\mu_i\in X^+(T)$ with $V^i/V^{i+1}\cong V_G(\mu_i).$ Similarly, such a filtration is called a \emph{good filtration} if for every $0\leqslant i\leqslant r,$ there exists a weight $\mu_i\in X^+(T)$ with $V^i/V^{i+1}\cong H^0(\mu_i).$ Finally, we call a $KG$-module \emph{tilting} if it admits both a good and a Weyl filtration.
\end{defin}
%------------------------------------------------|
%------------------------------------------------|

Modules with filtrations as above behave nicely with respect to tensor products and exterior (respectively, symmetric) powers, as recorded in the following result.

%-------------------------------------------------------------\
%------ Tensor Product of Modules with Good Filtrations ------\
%-------------------------------------------------------------\

\begin{prop}\label{Tensor product of modules with good filtrations}
If $U,$ $V$ are two $KG$-modules admitting good (respectively, Weyl) filtrations, then $U\otimes V$ also admits a good (respectively, Weyl) filtration. In addition,  if $W$ is a $KG$-module affording a good (respectively, Weyl) filtration, then each of $\SymP ^r W$ and $\Lambda ^r W$ admits a good (respectively, Weyl) filtration as well, for any $1\leqslant r < p.$
\end{prop}

\begin{proof}
The first general proof of the first assertion was given in \cite{Mathieu}, but it had already been proven in most cases in \cite{Donkin}. We refer to \cite[Proposition 2.2.5]{HM} for a proof of the second assertion.
\end{proof}
%-------------------------------------------------------------|
%-------------------------------------------------------------|

For $V_1,V_2$ two $KG$-modules, we identify $\Ext_G ^1 (V_2,V_1)$ with the set of equivalence classes of all short exact sequences $0\rightarrow V_1 \hookrightarrow V \twoheadrightarrow V_2 \rightarrow 0$ of $KG$-modules. To conclude this section, we record a result on the possible extensions between irreducible modules for $G.$

%-------------------------------------------------\
%------ Extensions between two irreducibles ------\
%-------------------------------------------------\
\begin{prop}\label{Extensions between two irreducibles}
Let $\lambda,\mu\in X^+(T),$ with $\mu \prec \lambda,$ and suppose that $[V_G(\lambda),L_G(\mu)]=0.$ Then $\Ext^1 _G(L_G(\lambda),L_G(\mu))=0.$
\end{prop} 
\begin{proof}
Let $\lambda,\mu \in X^+(T)$ be such that $\Ext^1 _G(L_G(\lambda),L_G(\mu))\neq 0.$ By \cite[II, Proposition 2.14]{Jantzen}, this translates to $\Hom_{KG}(\rad(\lambda),L_G(\mu))\neq 0.$ Consequently, there exists a non-zero surjective morphism of $KG$-modules $\phi:\rad (\lambda) \twoheadrightarrow L_G(\mu),$ so that $\ker \phi$ is maximal in $\rad (\lambda).$  Finding a composition series of $\ker (\phi)$ then yields a composition series for $V_G(\lambda),$ say $ V_G(\lambda) \supseteq \rad (\lambda) \supseteq \ker (\phi) \supseteq V_3\supseteq \ldots \supseteq V_r \supseteq 0.$ As $\rad(\lambda)/\ker (\phi) \cong L_G(\mu),$ we get that $[V_G(\lambda),L_G(\mu)]\neq 0,$ thus completing the proof.
\end{proof}
%-------------------------------------------------|
%-------------------------------------------------| 

%%%%%%%%%%%%%%%%%%%%%%%%%%%%%%%%%%%
\section{Main techniques}     %%%%%
%%%%%%%%%%%%%%%%%%%%%%%%%%%%%%%%%%%

In this section, we introduce two techniques (namely Proposition \ref{chT(omega)_as_a_sum_of_xi's} and Theorem \ref{How_to_determine_contributions} below) that shall be used in order to prove the three main results of the paper. The first result provides us with an upper bound (equal to zero for most dominant weights) for the number of times certain composition factors appear in a given Weyl module for $G.$

%%%%%%%%%%%%%%%%%%%%%%%%%%%%%%%%%%%%%%%%%%%%%%%%%%%%
\subsection{Extending a result of McNinch}     %%%%%
%%%%%%%%%%%%%%%%%%%%%%%%%%%%%%%%%%%%%%%%%%%%%%%%%%%%

Following the idea of \cite{McNinch}, we first investigate pairs $(V,\tau),$ where $V$ is a finite-dimensional, rational $KG$-module and $\tau\in \Lambda^+(V)$ satisfy  a certain set of properties.

%-------------------%
%----- McNinch -----%
%-------------------%
\begin{prop}\label{McNinch_Proposition}
Let $V$ be a finite-dimensional, rational $KG$-module, and let $\tau\in \Lambda^+(V)$ be a dominant weight of $V.$ Assume in addition that $V$ is tilting, that $\tau$ is the unique highest weight of $V,$ and that $\m_V(\tau)=1.$ Then there exists $\iota \in \Hom_{KG}(V_G(\tau),V)$ injective and  $\phi \in \Hom_{KG}(V, H^0(\tau))$ surjective. Furthermore, under those hypotheses, we have $\iota (\rad(\tau))\subseteq \ker(\phi).$
\end{prop} 

\begin{proof}
We refer the reader to \cite[Proposition 4.6.2]{McNinch} for a proof of the existence of $\iota$  and $\phi$ as in the statement of the proposition. For simplicity, we identify $V_G(\tau)$ with $\iota(V_G(\tau)) $ in the remainder of the proof.  Also write $N=\ker(\phi)\cap V_G(\tau),$ and denote by $\bar{\phi}:V_G(\tau)/N \hookrightarrow H^0(\tau)$ the injective morphism of $KG$-modules induced by $\phi\circ \iota.$ As $\rad (\tau)$ is the unique maximal submodule of $V_G(\tau),$ we have $N\subseteq \rad(\tau),$ and if $N\varsubsetneq \rad(\tau),$ then we get $0\varsubsetneq \bar{\phi}(\rad(\tau)/N) \subseteq \im (\bar{\phi}) \subseteq H^0(\lambda),$ a contradiction with $\soc (H^0(\tau))=L_G(\tau),$ as $\tau \notin \Lambda(\rad (\tau)).$ Therefore $N=\rad(\tau)$ and the proof is complete.
\end{proof}
%-------------------%
%-------------------%

%-------------------------------------------%
%----- Character of the tensor product -----%
%-------------------------------------------%
\begin{prop}\label{chT(omega)_as_a_sum_of_xi's}
Let $V$ be a $KG$-module as in the statement of Proposition \textnormal{ \ref{McNinch_Proposition}}, and let $(a_\mu)_{\mu \in X^+(T)}\subset \mathbb{Z}_{\geqslant 0}$ be such that $\ch V =\chi(\tau) + \sum_{\mu\in X^+(T)}{a_\mu \ch L_G(\mu)}.$ Then for every $\mu\in X^+(T) $ different from $\tau,$  we have 
$$
[V_G(\tau),L_G(\mu)]\leqslant a_\mu.
$$
\end{prop}

\begin{proof}
Let $\phi:V\twoheadrightarrow H^0(\tau)$ be as in Proposition \ref{McNinch_Proposition}, with $\rad (\tau) \subseteq \ker \phi,$ and write $M=\ker (\phi) / \rad (\tau).$ Also consider the short exact sequence $0\to \ker (\phi) \hookrightarrow V \twoheadrightarrow H^0(\tau)\to 0.$ Then one easily checks that $\chi(\tau)= \ch L_G(\tau) + \sum_{\mu\in  X^+(T)}{a_\mu \ch L_G(\mu)} -\ch M,$ so that  $[V_G(\tau),L_G(\mu)] = a_\mu -[M,L_G(\mu)]\leqslant a_\mu$ as desired. 
\end{proof}
%-------------------------------------------%
%-------------------------------------------%

We next illustrate Proposition \ref{chT(omega)_as_a_sum_of_xi's} with a concrete example, which shall prove useful later on in the paper. The result is somehow standard. (An alternative proof can be found in \cite[Lemma 8.6]{Se}, for instance.)

%------------------------------------------%
%-----     chT(\lambda) in type A     -----%
%------------------------------------------%
\begin{lem}\label{Tensor_products_in_type_A}
Let $G$ be a simple algebraic group of type $A_n $ $(n\geqslant 1)$ over $K.$ Also fix $1\leqslant j\leqslant n,$ and write $\lambda=\lambda_1+\lambda_j.$  Then $V_G(\lambda_1)\otimes V_G(\lambda_j)$ is tilting, and adopting the notation $\lambda_{n+1}=0,$ we have
$$
\ch V_G(\lambda_1)\otimes V_G(\lambda_j) =\chi(\lambda)+\chi(\lambda_{j+1}).
$$
In addition, if $\mu\in X^+(T)$ affords the highest weight of a composition factor of $V_G(\lambda),$ then $\mu =\lambda$ or $ \lambda_{j+1},$ and $[V_G(\lambda),L_G(\mu)]=1.$ 
\end{lem} 
\begin{proof}
First observe that each of $V_G(\lambda_1)$ and $V_G(\lambda_j)$ is irreducible by Lemma \ref{Dimension_of_various_irreducibles_in_type_A}, and hence both $KG$-modules are tilting. The first assertion  then follows from Proposition \ref{Tensor product of modules with good filtrations}. Also writing $T(\lambda)$ for $V_G(\lambda_1)\otimes V_G(\lambda_j),$ we observe that $\ch T(\lambda )$ is independent of $p$ and thus we may and shall assume $K$ has characteristic zero in  the remainder of the argument. An application of the Littlewood-Richardson formula \cite[16.4]{James} then yields the desired assertion on the character of $T(\lambda ).$ Finally, as $\lambda$ is the highest weight of $T(\lambda )$ and since $\m_{T(\lambda )}(\lambda)=1,$ an application of Proposition \ref{chT(omega)_as_a_sum_of_xi's} completes the proof. 
\end{proof}
%----------------------------------------%
%----------------------------------------%

%%%%%%%%%%%%%%%%%%%%%%%%%%%%%%%%%%%%%%%%%%%%%%%%%%%%%%%%%%%%%%%%%%%%%%%%%%%%%%%%
\subsection{A truncated version of the Jantzen $p$-sum formula}\label{Tools}  %%
%%%%%%%%%%%%%%%%%%%%%%%%%%%%%%%%%%%%%%%%%%%%%%%%%%%%%%%%%%%%%%%%%%%%%%%%%%%%%%%%

In this section, we introduce a few tools which shall be of use in order to better understand the composition factors of a given Weyl module for $G.$ Most of the underlying theory can be found in \cite[II, Sections 4, 5, or 8]{Jantzen}, to which we refer the reader for more details. Let $\rho$ denote the half-sum of all positive roots in $\Phi,$ or equivalently, the sum of all fundamental weights. Also for $\lambda,\mu\in X^+(T)$ such that $\mu \prec \lambda,$ define 
$$
\mbox{d}(\lambda,\mu)=2(\lambda+\rho,\lambda-\mu)-(\lambda-\mu,\lambda-\mu),
$$
as in \cite[Section 6]{Se}. The following corollary to the strong linkage principle \cite{Andersen} provides some insight on the possible composition factors of a given Weyl module for $G,$ in the case where $G$ is not of type $G_2$ and $p>2.$  We refer the reader to \cite[Proposition 6.2]{Se} for a proof.

%------------------------------------------------%
%------ Corollary to the Linkage principle ------%
%------------------------------------------------%
\begin{prop}\label{Corollary to The linkage principle}
Assume $p>2$ and let $G$ be a simple algebraic group of type different from $G_2.$  Also let $\lambda$ and $\mu$ be as above, and assume the inner product on $\Z \Phi$ is normalized so that long roots have length $1.$ If $\mu $ affords the highest weight of a composition factor of $V_G(\lambda),$ then 
$$
2\textnormal{d}(\lambda,\mu)\in p\Z.
$$
\end{prop}
%-----------------------------------------------%
%-----------------------------------------------%

For  $r\in \Z$ and $\alpha \in \Phi,$ we denote by $s_{\alpha,r}:X(T)\to X(T)$ the affine reflection on $X(T)$ defined by $
s_{\alpha,r}(\lambda)=s_{\alpha}(\lambda)+r\alpha,\mbox{ }\lambda\in X(T).
$  Also for $\ell$ a prime, set $\mathscr{W}_\ell$ equal to the subgroup of $\Aff (X(T))$ generated by all $s_{\alpha,n\ell},$ with $\alpha\in \Phi,$ $n\in \Z,$ and call $\mathscr{W}_\ell$ the \emph{affine Weyl group} associated to $G$ and $ \ell.$

The dot action introduced in Section \ref{Formal_character} can be extended to an action of $\mathscr{W}_\ell$ on $X(T)$ and $X(T)_{\R}$ in the obvious way, setting $w\cdot \lambda= w(\lambda+\rho)-\rho,$ $w\in \mathscr{W}_\ell,$ $\lambda\in X(T).$ Finally, for $\ell$  a prime number  and $m\in \Z,$ we write $\nu_\ell(m)$ to denote the greatest integer $r$ such that $\ell^r$ divides $m$ (adopting the notation $\nu_0(m)=0$ for every $m\in \Z$).  The following result provides a powerful tool for understanding Weyl modules.

%------------------------------------------------------------%
%------ ANOTHER WAY TO EXPRESS THE JANTZEN SUM FORMULA ------%
%------------------------------------------------------------%
\begin{prop}[The Jantzen $p$-sum formula]\label{Proposition:Sum_formula}
Let $\lambda\in X^+(T)$ be a dominant weight. Then there exists a filtration $V_G(\lambda)=V^0\supsetneq V^1 \supseteq \ldots \supseteq V^k \supseteq 0$ of $V_G(\lambda)$ such that $V^0/V^1 \cong L_G(\lambda)$ and  
\begin{equation}
\sum_{i=1}^{k}{\ch V^i} = -\sum_{\alpha \in \Phi^+}{\sum_{r=2}^{\langle \lambda+\rho, \alpha \rangle-1}{\nu_p(r)\det (w_{\alpha,r})\chi(\xi_{\alpha,r})}},
\label{Another version of Jantzen p-sum formula}
\end{equation}
where for $\alpha \in \Phi^+$ and $1<r<\langle \lambda+\rho,\alpha \rangle,$ $\xi_{\alpha,r}$ denotes the unique weight in $\mathscr{W}\cdot (\lambda-r\alpha) \cap \mathscr{D}$ and $w_{\alpha,r}$ is an element in $\mathscr{W}$ satisfying $w_{\alpha,r}\cdot (\lambda-r\alpha) = \xi_{\alpha,r}.$
\end{prop}

\begin{proof}
By \cite[II, 8.19]{Jantzen}, there exists a filtration  $V_G(\lambda)=V^0\supsetneq V^1 \supseteq \ldots \supseteq V^k \supseteq 0$ of $V_G(\lambda)$ such that $V^0/V^1 \cong L_G(\lambda)$ and 
$$
\sum_{i=1}^{k}{\ch V^i} = \sum_{\alpha \in \Phi^+}{\sum_{r=2}^{\langle \lambda+\rho, \alpha \rangle-1}{\nu_p(r)\chi(s_{\alpha,r}\cdot \lambda)}}.
$$ 
Fix $\alpha\in \Phi^+,$ $1<r<\langle \lambda+\rho,\alpha\rangle,$ and let  $w_{\alpha,r}\in \mathscr{W}$ and $\xi_{\alpha,r}\in \mathscr{D}$ be such that $w_{\alpha,r}\cdot \xi_{\alpha,r}=\lambda-r\alpha.$ (Such elements exist, since $\mathscr{D}$ is a fundamental domain for the dot action.)  A straightforward calculation yields  $s_{\alpha,r}\cdot \lambda = s_{\alpha}\cdot (\lambda-r\alpha),$ from which one deduces that $\chi(s_{\alpha,r} \cdot \lambda)  =\chi((s_\alpha w_{\alpha,r} )\cdot \xi_{\alpha,r}).$ An application of Lemma \ref{[Preliminaries] The Jantzen p-sum formula: Tools to compute chi(lambda)}  then completes the proof.
\end{proof}
%------------------------------------------------------------%
%------------------------------------------------------------%

We shall call a filtration of $V_G(\lambda)$ as in Proposition \ref{Proposition:Sum_formula} a \emph{Jantzen filtration} of $V_G(\lambda).$ Let us then fix such a filtration $V_G(\lambda)=V^0\supsetneq V^1 \supseteq \ldots \supseteq V^k \supseteq 0$ in the remainder of the section. Also, following  \cite[II, 8.14]{Jantzen}, we write $\nu^c (T_{\lambda})$ to denote the expression \eqref{Another version of Jantzen p-sum formula}. As  $\{\chi(\lambda)\}_{\lambda\in X^+(T)}$ forms a $\Z$-basis of $\Z[X(T)]^{\mathscr{W}} $  (see \cite[Remark II.5.8]{Jantzen}, for instance), there exists $(a_\nu)_{\nu\in X^+(T)}\subset \mathbb{Z}$  such that   
\begin{equation}
\nu^c(T_\lambda)=\sum_{\nu \in X^+(T)}{a_\nu \chi(\nu)}.
\label{nu^c(T_lambda)_in_terms_of_chi(nu)'s}
\end{equation}

Consider a $T$-weight $\mu \in X(T)$ with $\mu \prec \lambda.$ In what follows, we  introduce a ``truncated'' version of the character $\nu^c(T_\lambda),$ which shall prove useful in computations. Define   
\begin{equation}
\nu_\mu^c(T_\lambda)= \sum_{\substack{\nu \in X^+(T) \\ \mu \preccurlyeq \nu \prec \lambda}}{a_\nu \chi_\mu (\nu)},
\label{nu_mu^c(T_lambda)_in_terms_of_chi_mu(nu)'s}
\end{equation}
where the $a_\nu$ $(\nu\in X^+(T))$ are as in \eqref{nu^c(T_lambda)_in_terms_of_chi(nu)'s}.  Finally, the latter decomposition yields  
\begin{equation}
\nu_\mu^c(T_\lambda)= \sum_{\substack{\xi \in X^+(T) \\ \mu \preccurlyeq \xi \prec \lambda}}{b_\xi \ch L_G(\xi)},
\label{nu_mu^c(T_lambda)_in_terms_of_ch(xi)'s}
\end{equation}
for some $b_\xi  \in \Z$ $,\xi\in X^+(T).$ 

The following proposition provides some insight on how the truncated $p$-sum formula \eqref{nu_mu^c(T_lambda)_in_terms_of_ch(xi)'s}  can be used in order to determine the possible composition factors of $V_G(\lambda),$ together with an upper bound for their multiplicity.

%--------------------------------------------------------------------%
%------ How to get information using the Jantzen p-sum formula ------%
%--------------------------------------------------------------------%
\begin{prop}\label{Insight_on_possible_composition_factors}
Let $\lambda\in X^+(T)$ and consider a $T$-weight $\mu \prec \lambda.$ Also let $\xi \in X^+(T)$ be a dominant weight such that $\mu \preccurlyeq \xi \prec \lambda.$ Then $\xi$ affords the highest weight of a composition factor of $V_G(\lambda)$ if and only if $b_\xi\neq 0$ in \eqref{nu_mu^c(T_lambda)_in_terms_of_ch(xi)'s}. Also  $[V_G(\lambda),L_G(\xi)]\leqslant b_\xi.$
\end{prop}

\begin{proof}
By definition, we have $\nu^c(T_\lambda)=\ch V,$ where $V=V^1\oplus \cdots \oplus V^k.$  We first claim that $\xi$ affords the highest weight of a composition factor of $V_G(\lambda)$ if and only if $[V,L_G(\xi)]\neq 0.$ Indeed, $V^i \subsetneq V_G(\lambda)$ for $1 \leqslant i \leqslant k,$ and since $V_G(\lambda)/V^1 \cong L_G(\lambda),$ we have  
$$
0\leqslant [V_G(\lambda),L_G(\xi)] \leqslant [V,L_G(\xi)].
$$
In particular if $\xi$ affords the highest weight of a composition factor $V_G(\lambda),$ then $[V,L_G(\xi)] \neq 0.$ Conversely, assume $[V,L_G(\xi)] \neq 0.$ Since $V^0/V^1\cong L_G(\lambda)$ and $\xi\neq \lambda,$ there exists $1\leqslant i\leqslant k$ such that $[V^i,L_G(\xi)]\neq 0,$ and as $V^i\subseteq V_G(\lambda),$ the claim holds. In order to conclude, it remains to show that    $[V,L_G(\xi)]=b_\xi,$ where $b_\xi$ is as in  \eqref{nu_mu^c(T_lambda)_in_terms_of_ch(xi)'s}. Since every $T$-weight of $V_G(\lambda)$ (and hence of $V$) is under $\lambda,$ there exist integers $c_\nu$ such that 
$$
\nu_\mu^c(T_\lambda) = \nu^c(T_\lambda)-\sum_{\substack{\nu \in X^+(T)\\ \mu \not\preccurlyeq \nu \prec\lambda }}{c_\nu \ch L_G(\nu)}.
$$
In particular, this shows that $b_\xi =[V,L_G(\xi)]$ as desired, thus completing the proof.
\end{proof}
%--------------------------------------------------------------------%
%--------------------------------------------------------------------%

Fix $\mu\in X^+(T)$ with $\mu\prec \lambda.$ For $\nu\in X^+(T),$ we call the coefficient $a_\nu$ in \eqref{nu_mu^c(T_lambda)_in_terms_of_chi_mu(nu)'s} the \emph{contribution} of $\nu$ to $\nu^c_\mu(T_\lambda),$ and we say that $\nu$ \emph{contributes} to $\nu^c_\mu(T_\lambda)$ if its contribution is non-zero. Now applying Proposition \ref{Insight_on_possible_composition_factors} to a given triple  $\mu\preccurlyeq \xi \prec \lambda$ requires the knowledge of the contribution of $\nu$ to $\nu_\mu^c(T_\lambda)$ for each dominant $T$-weight $\xi \preccurlyeq \nu \prec \lambda.$  In certain cases, knowing whether or not a given $T$-weight contributes to $\nu^c_\mu(T_\lambda)$  can be easily determined, as the following result shows.

%--------------------------------------------------------------------------------------------------------------\
%------ No weight having multiplicity 1 in V(lambda) can give a contribution to the Jantzen p-sum formula -----\
%--------------------------------------------------------------------------------------------------------------\
\begin{lem}\label{no_contribution_for_certain_weights}
Let $\lambda,$ $\mu$ and $\nu$ be as above, with $\nu$ maximal  with respect to the partial order   $\preccurlyeq$, such that $\nu$ contributes to $\nu^c_{\mu}(T_{\lambda}).$ Then $\nu$ affords the highest weight of a composition factor of $V_G(\lambda).$  
\end{lem} 

\begin{proof}
Write $\Lambda^+=\{\xi\in X^+(T):\mu \preccurlyeq \xi \preccurlyeq \lambda,~\nu \not\preccurlyeq \xi\}.$ By maximality of $\nu,$ there exists $a_\nu\in \mathbb{Z}^*$ and $(a_\xi)_{\xi\in \Lambda^+}\subset \Z$ such that 
\begin{equation}
\nu_\mu^c(T_\lambda)=a_\nu \chi_\mu(\nu) + \sum_{\xi\in \Lambda^+}{a_\xi \chi_\mu(\xi)}.
\label{Maximal_weight_contributing}
\end{equation}
As $\sigma \preccurlyeq \xi$ for all $\sigma \in \Lambda(\xi),$ $\xi\in \Lambda^+,$ we get that $\nu$ is not a weight in $\Lambda(\xi),$ $\xi\in \Lambda^+,$ and hence cannot afford the highest weight of a composition factor of $V_G(\xi),$ $\xi\in \Lambda^+.$ Therefore, rewriting $\chi_\mu(\nu)$ and each $\chi_\mu(\xi)$ of \eqref{Maximal_weight_contributing} in terms of characters of irreducibles (recall that  $[V_G(\nu),L_G(\nu)]=1$)  yields the existence of a tuple $(b_\xi)_{\xi \in \Lambda^+}\subset \Z$ such that $\nu_\mu^c(T_\lambda)=a_\nu \ch L_G(\nu) + \sum_{\xi\in \Lambda^+}{b_\xi\ch L_G(\xi)}.$ An application of Proposition  \ref{Insight_on_possible_composition_factors} then completes the proof.
\end{proof}
%--------------------------------------------------------------------------------------------------------------|
%--------------------------------------------------------------------------------------------------------------|

Fix $\nu\in X^+(T),$ and recall from \cite{Bourbaki} the description of the simple roots and fundamental weights for $T$ in terms of a basis $\{\varepsilon_1,\ldots,\varepsilon_{d_\Phi}\}$ for a Euclidean space $E$ of dimension $d_\Phi.$   Following the idea of \cite{McNinch}, for  $\alpha \in \Phi^+$ and $r\in \mathbb{Z}_{\geqslant 0}$ such that $1<r<\langle \lambda+\rho, \alpha\rangle$, we write $\lambda+\rho-r\alpha=a_1\varepsilon_1+\cdots+a_{d_\Phi}\varepsilon_{d_\Phi},$ as well as $\nu +\rho =b_1\varepsilon_1+\cdots+b_{d_\Phi}\varepsilon_{d_\Phi}.$ Also, we set $A_{\alpha,r}=(a_j)_{j=1}^{d_\Phi}\in \mathbb{Q}^{d_\Phi}$ and $B_{\nu}= (b_j)_{j=1}^{d_\Phi}\in \mathbb{Q}^{d_\Phi}.$  The action of the Weyl group $\mathscr{W}$ of $G$ on the basis $\{\varepsilon_1,\ldots,\varepsilon_{d_\Phi}\}$ is described in \cite{Bourbaki}, and extends to an action of $\mathscr{W}$ on $\mathbb{Q}^{d_\Phi}$ in the obvious way. (We write $w\cdot A$ for $w\in \mathscr{W},$ $A\in \mathbb{Q}^{d_\Phi}.$) Define the \emph{support} of an element $z\in \Z\Phi$ to be the subset $\mbox{supp}(z)$ of $\Pi$ consisting of those simple roots $\alpha$ such that $c_{\alpha}\neq 0$ in the decomposition $z=\sum{c_{\alpha}\alpha}.$ Also for  $w\in \mathscr{W},$ we write $\det (w)$ for the determinant of $w$ as an invertible linear transformation of $X(T)_{\R}.$ The following result is  our main tool for determining the contribution of $\nu$ to $\nu^c_\mu(T_\lambda),$ for each weight $\nu\in X^+(T)$ with $\mu\preccurlyeq \nu \prec\lambda.$

%--------------------------------------------------%
%-----     How to determine contributions     -----%
%--------------------------------------------------%
\begin{thm}\label{How_to_determine_contributions}
Let $\lambda\in X^+(T),$ and fix a Jantzen filtration $V_G(\lambda)\supsetneq V^1\supseteq V^2 \supseteq \ldots \supseteq V^k  \supseteq 0$ of $V_G(\lambda).$ Also consider a weight $\mu\in X(T)$ with $\mu\prec\lambda,$ and let $\nu\in X^+(T)$ be such that $\mu\preccurlyeq \nu\prec\lambda.$ Finally, write $I_\nu=\{(\alpha,r)\in \Phi^+\times [2,\langle \lambda+\rho,\alpha\rangle]:\textnormal{supp}(\alpha)=\textnormal{supp}(\lambda-\nu), B_\nu \in  \mathscr{W} \cdot A_{\alpha,r}\},$ and for each pair $(\alpha,r)\in I_\nu,$ choose $w_{\alpha,r}\in \mathscr{W}$ such that $w_{\alpha,r}\cdot A_{\alpha,r}=B_\nu.$ Then the contribution of $\nu$ to $\nu_\mu^c(T_\lambda)$ is given by 
$$
-\sum_{(\alpha,r)\in I_\nu}{\nu_p(r)\det (w_{\alpha,r})}.
$$ 
\end{thm}

\begin{proof}
By Proposition \ref{Proposition:Sum_formula}, the contribution of $\nu$ to the truncated Jantzen $p$-sum formula $\nu_\mu^c(T_\lambda)$ is given by  
$$
-\sum_{\alpha\in \Phi^+}\sum_{r=2}^{\langle \lambda+\rho,\alpha\rangle -1}{\nu_p(r) \det (w_{\alpha,r})},
$$
where for $\alpha\in \Phi^+$ and $ 2\leqslant r\leqslant \langle \lambda+\rho,\alpha\rangle -1,$ the element $w_{\alpha,r}$ is either zero (if  $\nu \notin \mathscr{W}\cdot (\lambda-r\alpha)$), or a chosen element in $\mathscr{W}$ satisfying $w_{\alpha,r}\cdot   (\lambda-r\alpha)=\nu.$ In addition, observe that by \cite[Lemma 4.5.6]{McNinch}, the latter can only occur if $\lambda-r\alpha$ and $\nu$ have the same support. Finally, identifying the action of $\mathscr{W}$ on $X(T)$ with that of $\mathscr{W}$ on $\mathbb{Q}^{d_\Phi}$ as above completes the proof. 
\end{proof}
%--------------------------------------------------%
%--------------------------------------------------% 

%%%%%%%%%%%%%%%%%%%%%%%%%%%%%%%%%%%%%%%%%%%%%%%%%%%%%%%%%%%%%%
%%%%%%%%%%%%%%%%%%%%%%%%%%%%%%%%%%%%%%%%%%%%%%%%%%%%%%%%%%%%%%
\section{The $B_n$-case $(n\geqslant 2)$}\label{The_Bn_case}	%%%%%
%%%%%%%%%%%%%%%%%%%%%%%%%%%%%%%%%%%%%%%%%%%%%%%%%%%%%%%%%%%%%%
%%%%%%%%%%%%%%%%%%%%%%%%%%%%%%%%%%%%%%%%%%%%%%%%%%%%%%%%%%%%%%

Let $K$ be an algebraically closed field of characteristic $p\neq 2,$ and let $Y$  a simply connected, simple algebraic group of type $A_{2n}$ $(n\geqslant 2)$ over $K.$ Consider a subgroup $G$ of type $B_n,$ embedded in the usual way, as the stabilizer of a non-degenerate quadratic form on the natural module for $Y.$   Fix a Borel subgroup $B_Y=  U_Y T_Y$ of $Y,$ where $T_Y$ is a maximal torus of $Y$ and $U_Y$ is the unipotent radical of $B_Y,$ let $\Pi(Y) = \{\alpha_1,\ldots,\alpha_{2n}\}$ denote a corresponding base of the root system $\Phi(Y)=\Phi^+(Y)\sqcup \Phi^-(Y)$ of $Y,$ and let $\{\lambda_1,\ldots,\lambda_{2n}\}$ be the set of fundamental dominant weights for $T_Y$ corresponding to our choice of base $\Pi(Y).$ Also set $T=T_Y\cap G,$ $B=B_Y\cap G,$ so that $T$ is a maximal torus of $G,$ and $B$ is a Borel subgroup of $G$ containing $T.$ Let $\Pi(G)=\{\beta_1,\ldots,\beta_n\}$ be the corresponding base for the root system $\Phi(G)=\Phi^+(G)\sqcup \Phi^-(G)$ of $G,$ and let $\varpi_1,\ldots,\varpi_n$ denote the associated fundamental weights. Here the root restrictions are given by $\alpha_i|_{T}=\alpha_{2n-i+1}|_{T}=\beta_i$ for $1\leqslant i\leqslant n.$ Finally, using \cite[Table 1, p.69]{Humphreys1} and the fact that $\lambda_1|_{T}=\varpi_1$ yields
\begin{equation}
\lambda_i|_{T}=\lambda_{2n-i+1}|_{T}=\varpi_i, \mbox{ } \lambda_n|_{T}=\lambda_{n+1}|_{T}= 2\varpi_n, ~1\leqslant i\leqslant n-1.
\label{[Dn<A2n-1] Weight restrictions for B}
\end{equation} 

In this section, we show how to obtain the tables \ref{Table_B}, \ref{Table_of_dimensions_main_result_type_B}, and \ref{Restriction_of_V_Y(lambda)_to_G_of_type_B} in Theorem \ref{Main_Result_1}, Corollary \ref{Introduction: corollary on the dimension (B)}, and Proposition \ref{Restrictions_of_irreducibles_of_type_A}, respectively. In order to do so, we   rely as much as possible on the embedding of $G$ in $Y$ described above, proceeding in the following steps: we start by computing the formal character of the restriction to $G$ of certain Weyl modules for $Y$ (see Proposition  \ref{Type_B:restriction_in_characteristic_zero_0<j<n+1}), and then deduce the character of certain tensor products of irreducible Weyl modules for $G$ (see Lemmas \ref{chT(omega)_for_B} and \ref{Another_chT(omega)_for_B}). Applying Proposition \ref{chT(omega)_as_a_sum_of_xi's} shall then yield an upper bound for the multiplicities of the possible composition factors of the Weyl modules $V_G(\varpi),$ in the case where $\varpi$ is as in the statement of Theorem \ref{Main_Result_1}. Finally, using Theorem \ref{How_to_determine_contributions}, we compute various contributions to the Jantzen $p$-sum formula in each case, and  we conclude using Proposition \ref{Insight_on_possible_composition_factors}. The proofs of Corollary \ref{Introduction: corollary on the dimension (B)} and Proposition \ref{Restrictions_of_irreducibles_of_type_A} are also given at the end on the section.

%%%%%%%%%%%%%%%%%%%%%%%%%%%%%%%%%%%%%%%%%%%%%%%%%%%%%%%%%%%%%%%%%%  
\subsection{Restriction of certain Weyl modules for $Y$}     %%%%%
%%%%%%%%%%%%%%%%%%%%%%%%%%%%%%%%%%%%%%%%%%%%%%%%%%%%%%%%%%%%%%%%%%

We start our investigation by showing that if $V$ is a $KY$-module with unique  highest weight $\lambda\in X^+(T_Y),$ then every $T$-weight of $V$ is under the restriction of $\lambda$ to $T.$ Hence the ordering of $T_Y$-weights is preserved when restricting to $G.$

%----------------------------------------------------%
%----- Restriction of highest weights in type B -----%
%----------------------------------------------------%
\begin{lem}\label{[B]restriction_of_highest_weight}
Let $\lambda\in X^+(T_Y)$ be a dominant weight, and let $V$ be a $KY$-module with unique highest weight $\lambda.$ Then every $T$-weight $\xi$ of $V$  satisfies $\xi \preccurlyeq \lambda|_{T}.$
\end{lem}
\begin{proof}
Let $\xi \in X(T)$ be a weight of $V|_G.$ Then there exists a $T_Y$-weight  $\mu $ of $V$ such that $\mu|_{T}=\xi.$ Since $V$ has unique highest weight, there exist $c_1,\ldots,c_{2n}\in \mathbb{Z}_{\geqslant 0}$ such that $\mu=\lambda-\sum_{r=1}^{2n}{c_r\alpha_r}.$ Therefore $\xi= \lambda|_{T}-\sum_{r=1}^{2n}{c_r\alpha_r|_{T}} $ and the assertion follows from the root and weight restrictions in \eqref{[Dn<A2n-1] Weight restrictions for B}.
\end{proof}
%----------------------------------------------------%
%----------------------------------------------------%

We next investigate the formal character of the restriction to $G$ of the Weyl module  $V_Y(\lambda_1+\lambda_j),$ where $1\leqslant j\leqslant 2n.$

%----------------------------------------------------------------|
%-----     Restriction in characteristic zero in type B     -----|
%----------------------------------------------------------------|
\begin{prop}\label{Type_B:restriction_in_characteristic_zero_0<j<n+1}
Let $1\leqslant j\leqslant 2n,$ and consider the dominant $T_Y$-weight $\lambda=\lambda_1+\lambda_j\in X^+(T_Y).$ Also set $\varpi=\lambda|_{T}$ and adopt the notation $\varpi_0=0.$ Then 
$$
\ch V_Y(\lambda)|_G =
\begin{cases}
\chi(\varpi) + \chi(\varpi_{j-1})		& \mbox{ if $1\leqslant j\leqslant n;$}	\cr 
\chi(\varpi) + \chi(2\varpi_n)			& \mbox{ if $j=n+1,n+2;$}		\cr
\chi(\varpi) + \chi(\varpi_{2n-j+2})		& \mbox{ otherwise.}				\cr
\end{cases}
$$

\end{prop}
\begin{proof}
Write $V=V_Y(\lambda)$ and first notice that $\ch V|_G$ is independent of $p.$ Hence  we may and shall assume $K$ has characteristic zero in the remainder of the proof. In the case where $1\leqslant j\leqslant n,$ an application of \cite[Proposition 1.5.3]{Koike} yields $V_Y(\lambda)|_G\cong V_G(\varpi)\oplus V_G(\varpi_{j-1}),$ from which the result follows in this situation. Next assume $j=n+1.$ Here the weights $\lambda-(\alpha_1+\cdots+\alpha_n),$ $\lambda-(\alpha_1+\cdots+\alpha_r+\alpha_{n+1}+\cdots +\alpha_{2n -r})$  $(1\leqslant r\leqslant n-1),$ and $\lambda-(\alpha_{n+1}+\cdots+\alpha_{2n})$ all restrict to $\varpi'=2\varpi_n\in X^+(T).$ Therefore the latter occurs in a second composition factor of $V_Y(\lambda)|_G,$ whose highest weight $\mu\in X^+(T)$ satisfies $\varpi'\preccurlyeq \mu \preccurlyeq \varpi$ by Lemma  \ref{[B]restriction_of_highest_weight}. The only possibility is that $\varpi'$ itself affords the highest weight of a composition factor. Applying Theorem \ref{Weyl's formula}, Table \ref{Various_dimensions_in_type_A}, and Table \ref{Various_dimensions_in_type_B} then yields $\dim V_Y(\lambda)=\dim V_G(\varpi) + \dim V_G(2\varpi_n),$ from which the desired result follows in this case as well. The remaining cases can be dealt with in a similar fashion, hence the details are omitted here.
\end{proof}
%----------------------------------------------------------------|
%----------------------------------------------------------------|

%%%%%%%%%%%%%%%%%%%%%%%%%%%%%%%%%%%%%%%%%%%%%%%%%%%%%%%%%%%%%%%%%%
\subsection{Formal character of various tensor products}     %%%%%
%%%%%%%%%%%%%%%%%%%%%%%%%%%%%%%%%%%%%%%%%%%%%%%%%%%%%%%%%%%%%%%%%%

We next determine the formal character of the tensor product $V_G(\varpi_1)\otimes V_G(\varpi_j)$ for $1\leqslant j\leqslant n,$ as well as of the formal character of the tensor product $V_G(\varpi_1)\otimes V_G(2\varpi_n).$ Observe that since $p\neq 2,$ each of the considered Weyl modules is irreducible, and hence is tilting.

%-------------------------------------%
%-----     chT(\varpi) for B     -----%
%-------------------------------------%
\begin{lem}\label{chT(omega)_for_B}
Let $1\leqslant j\leqslant n,$ and consider the dominant $T$-weight $\varpi=\varpi_1+\varpi_j\in X^+(T).$ Also set $\varpi_0=\varpi_{n+1}=0$ and write $T(w)$ for the tensor product $V_G(\varpi_1)\otimes V_G(\varpi_j).$ Then $T(\varpi)$ is tilting and its formal character is given by 
$$
\ch T(\varpi)=\chi(\varpi)+(1-\delta_{j,n})\chi(\varpi_{j-1})+ (1-\delta_{j,1})\chi(\varpi_{j+1}+(\delta_{j,n-1} + \delta_{j,n})\varpi_n).
$$
\end{lem}

\begin{proof}
By Lemma \ref{Dimension_of_various_irreducibles_in_type_B}, both $V_G(\varpi_1)$ and $V_G(\varpi_j)$ are irreducible $KG$-modules, and hence $T(\varpi)$ is tilting by Proposition \ref{Tensor product of modules with good filtrations}. Also $\ch T(\varpi)$ is independent of $p,$ so we may and shall assume $K$ has characteristic zero in the remainder of the proof. By \cite[Theorem 1, Table 1 $(\textnormal{I}_2,\textnormal{I}_3)$]{Se} together with Lemma \ref{Tensor_products_in_type_A}, we successively get 
\begin{equation*}
T(\varpi)	\cong	\left(V_Y(\lambda_1)\otimes V_Y(\lambda_j)\right)|_G	\cong 	 V_Y(\lambda)|_G \oplus V_Y(\lambda_{j+1})|_G.
\end{equation*}

Now if $1\leqslant j <n,$ then applying \eqref{[Dn<A2n-1] Weight restrictions for B} yields the restrictions $\lambda|_{T}=\varpi$ and $\lambda_{j+1}|_{T} = \varpi_{j+1}+\delta_{j,n-1}\varpi_n.$ Therefore Proposition \ref{Type_B:restriction_in_characteristic_zero_0<j<n+1} yields $V_Y(\lambda)|_G\cong V_G(\varpi) \oplus V_G(\varpi_{j-1}),$ while $V_Y(\lambda_{j+1})|_G \cong  V_G(\varpi_{j+1}+\delta_{j,n}\varpi_n)$ by \cite[Theorem 1, Table 1 $(\textnormal{I}_2,\textnormal{I}_3)$]{Se}. The assertion thus holds in this situation and so it remains to consider the case where $j=n.$ Here applying Theorem \ref{Weyl's formula} yields 
\begin{equation}
\dim T(\varpi) = \dim V_G(\varpi) + 2^n > \dim V_G(\varpi),
\label{dim T(omega)=dim V(omega) + dim V(omega_n) in type B}
\end{equation} 
showing the existence of a second composition factor of $T(\varpi).$ Now $\varpi$ is the unique highest weight of $T(\varpi)$ and so an application of Lemma \ref{[B]restriction_of_highest_weight} yields $\Lambda^+(T(\varpi))=\{\varpi,\varpi_n\}.$ As $\m_{T(\varpi)}(\varpi)=\m_{V_G(\varpi)}(\varpi)=1,$ the weight $\varpi$ cannot afford the highest weight of a second composition factor of $T(\varpi),$ thus forcing $[T(\varpi),V_G(\varpi_n)]>0.$ Finally,  $\dim V_G(\varpi_n)=2^n $ by Table \ref{Various_dimensions_in_type_B}, and hence \eqref{dim T(omega)=dim V(omega) + dim V(omega_n) in type B} completes the proof.
\end{proof}
%-------------------------------------% 
%-------------------------------------%

%---------------------------------------------%
%-----     Another chT(\varpi) for B     -----%
%---------------------------------------------%
\begin{lem}\label{Another_chT(omega)_for_B}
Consider the $T$-weight $\varpi=\varpi_1+2\varpi_n\in X^+(T),$ and write $T(w)=V_G(\varpi_1)\otimes V_G(2\varpi_n).$ Then $T(\varpi)$ is tilting and its formal character is given by 
$$
\ch T(\varpi)=\chi(\varpi) + \chi(\varpi_{n-1})+\chi(2\varpi_n).
$$
\end{lem}
\begin{proof}
By Lemma  \ref{Dimension_of_various_irreducibles_in_type_B}, both $V_G(\varpi_1)$ and $V_G(2\varpi_n)$ are irreducible $KG$-modules, and hence $T(\varpi)$ is tilting by Proposition \ref{Tensor product of modules with good filtrations}. Also $\ch T(\varpi)$ is independent of $p,$ so we may and shall assume $K$ has characteristic zero in the remainder of the proof. By \cite[Theorem 1, Table 1 $(\textnormal{I}_2,\textnormal{I}_3)$]{Se} together with Lemma \ref{Tensor_products_in_type_A}, we successively get 
\begin{equation*}
T(\varpi)	\cong	\left(V_Y(\lambda_1)\otimes V_Y(\lambda_n)\right)|_G	\cong 	 V_Y(\lambda)|_G \oplus V_G(2\varpi_{n}).
\end{equation*}
Now applying \eqref{[Dn<A2n-1] Weight restrictions for B} yields the restriction $\lambda|_{T}=\varpi$ and so $V_Y(\lambda)|_G\cong V_G(\varpi) \oplus V_G(\varpi_{n-1})$ by  Proposition \ref{Type_B:restriction_in_characteristic_zero_0<j<n+1}, thus completing the proof.
\end{proof}
%---------------------------------------------%
%---------------------------------------------%

%%%%%%%%%%%%%%%%%%%%%%%%%%%%%%%%%%%%%%%%%%%%%%%%%%%%%%%%%%%%%%%%%%%%%%%%%%%%%%%%%
\subsection{Various contributions to the truncated Jantzen $p$-sum formula}  %%%%
%%%%%%%%%%%%%%%%%%%%%%%%%%%%%%%%%%%%%%%%%%%%%%%%%%%%%%%%%%%%%%%%%%%%%%%%%%%%%%%%%

In this section, we compute certain contributions to the truncated Jantzen formula for some Weyl modules $V_G(\varpi),$ where $\varpi$ is as in the first column of Table \ref{Table_B}, starting by the case where $\varpi=2\varpi_1.$

%-----------------------------------------------%
%----- Type B: Sum formula for (2,0,...,0) -----%
%-----------------------------------------------%
\begin{prop}\label{Type_B:Sum_formula_for_(2,0,...,0)} 
Assume $\varpi=2\varpi_1,$ and consider the zero weight $\mu=0.$ Let $V_G(\varpi)=V^0\supsetneq V^1\supseteq \ldots \supseteq V^k \supseteq 0$ be a Jantzen filtration of $V_G(\varpi).$ Then 
$$
\nu^c_\mu(T_\varpi)=\nu_p(2n+1)\ch L_G(\mu).
$$
\end{prop}

\begin{proof}
We start by computing all contributions to $\nu_\mu^c(T_\varpi).$ Here the  dominant $T$-weights $\nu\in X^+(T)$ such that  $\mu\preccurlyeq \nu \prec \varpi$ are $\varpi-\beta_1,$ $\varpi-(\beta_1+\cdots+\beta_n),$ and $\mu$ itself. Now $\varpi-\beta_1$ and $\varpi-(\beta_1+\cdots+\beta_n)$    have multiplicity $1$ in $V_G(\varpi),$ and hence none of them can afford of $V_G(\varpi)$ the highest weight of a composition factor by \cite{Pr}. Recursively applying Lemma \ref{no_contribution_for_certain_weights} then shows that those same weights cannot contribute to $\nu_\mu^c(T_\varpi).$ Therefore  it remains to compute the contribution of $\mu,$ starting by determining all pairs $(\beta,r)\in I_\mu$ as in Theorem \ref{How_to_determine_contributions}. A straightforward computation yields 
$$
B_\mu = \tfrac{1}{2}(2n-1,2n-3,\ldots,3,1),
$$
and since  $\varpi-\mu$ has support $\Pi,$ we get that  $\beta\in \{\varepsilon_1,\varepsilon_1+\varepsilon_2,\varepsilon_1+\varepsilon_3,\ldots,\varepsilon_1+\varepsilon_n\}$ by definition of $I_\mu.$ Recall from \cite[Planche II]{Bourbaki} that $\mathscr{W}$ acts by all permutations and sign changes of the $\varepsilon_i.$ Also, one checks that for  $2 \leqslant \ell\leqslant n$ and $r\in \mathbb{Z},$  we have 

$$
A_{\varepsilon_1+\varepsilon_\ell,r}= B_\mu +(2-r,\underbrace{0,\ldots,0}_{\ell-2},-r,\underbrace{0,\ldots,0}_{n-\ell}).
$$
Consequently $B_\mu \in \mathscr{W}\cdot A_{\varepsilon_1+\varepsilon_\ell,r}$ if and only if $\{| 2(n-r)+3| , |2(n-\ell-r)+1| \}=\{2n-1,2(n-\ell)+1\}.$  We thus study each possibility separately and show that in each case, no weight in  $\mathscr{W}\cdot(\varpi-\mathbb{Z}(\varepsilon_1+\varepsilon_\ell))$ can contribute to $\nu^c_\mu(T_\varpi).$

\begin{enumerate}
\item If $2(n-r)+3=2n-1,$ then $ r=2$ and hence $\nu_p(r)=0,$ since $p\neq 2$ by assumption. Hence the weight $\varpi-2(\varepsilon_1+\varepsilon_\ell)$ cannot contribute to $\nu^c_\mu(T_\varpi)$ in this situation.
\item If $2(n-r)+3=-2n+1,$ then $r=2n+1,$ so that $|2(n-\ell-r)+1| = 2(n+\ell)+1.$ As it is impossible for the latter to be equal to $2(n-\ell)+1,$ we get the desired assertion in this case.
\item If $2(n-r)+3=2(n-\ell)+1,$ then $r=\ell+1$ and $|2(n -\ell-r)+1|=|2(n-2\ell)-1|.$ The latter cannot equal $2n-1,$ thus showing the assertion in this case as well.
\item If $2(n-r)+3=-2(n-\ell)-1,$ then $r=2n-\ell+2,$ in which case $|2(n- \ell-r)+1|=2n+3\neq 2n-1.$
\end{enumerate} 
Therefore a contribution to 
$\nu^c_\mu(T_\varpi)$ can only come from the situation where  $\beta=\varepsilon_1,$ which we assume holds in the remainder of the proof. Here for $r\in \Z,$ we have
$$
A_{\beta,r}= B_\mu+(2-r,0,\ldots,0),
$$
from which one deduces that  $B_\mu \in \mathscr{W}\cdot A_{\beta,r}$ if and only if $ |2(n-r)+3| =2n-1.$ The latter equality is satisfied if and only if   $r=2$ or $2n+1,$ and since $p\neq 2,$ the weight $\varpi-2\beta$ cannot contribute to the $p$-sum  $\nu_\mu^c(T_\varpi).$ Using the action of $\mathscr{W}$ described above, one checks that  $\mu=s_{\varepsilon_1} \cdot (\varpi-(2n+1)\varepsilon_1),$ so that   $\nu^c_\mu(T_\varpi)=\nu_p(2n+1)\chi _\mu (\mu).$ Finally, as $\chi_\mu(\mu)=\ch L_G(\mu),$ the proof is complete.  
\end{proof}
%-----------------------------------------------%
%-----------------------------------------------%

Proceeding as in the proof of Proposition \ref{Type_B:Sum_formula_for_(2,0,...,0)}, we next compute the formal character $\nu^c_\mu(T_\varpi)$ in the situation where $\varpi=\varpi_1+\varpi_j$ for some $1<j<n,$ and $\mu=\varpi_{j-1}.$ Recall that the case $j=2$ was dealt with in \cite[Lemma 4.5.7]{McNinch}.

%---------------------------------------------------------%
%----- Type B: Sum formula for (1,0,...,0,1,0,...,0) -----%
%---------------------------------------------------------% 
\begin{prop}\label{Type_B:Sum_formula_for_(1,0,...,0,1,0,...,0)} 
Assume $\varpi= \varpi_1 + \varpi_j$ for some $1<j<n,$ and write $\mu=\varpi_{j-1}.$ Let $V_G(\varpi)=V^0\supsetneq V^1\supseteq \ldots \supseteq V^k \supseteq 0$ be a Jantzen filtration of $V_G(\varpi).$  Then
$$
\nu^c_\mu(T_\varpi)=\nu_p(j+1)\ch L_G(\varpi_{j+1} + \delta_{j,n-1}\varpi_n) + \nu_p(2n-j+2)\ch L_G(\mu).
$$
\end{prop}
 
\begin{proof}
We refer the reader to \cite[Lemma 4.5.7]{McNinch} for a proof in the case where $j=2$ and hence assume $2<j<n$ in what follows.  Here the dominant $T$-weights $\nu\in X^+(T)$ such that $\mu \preccurlyeq \nu \prec \varpi$ are $\nu_1=\varpi_{j+1} + \delta_{j,n-1}\varpi_n,$ $\nu_2=\varpi_1+\varpi_{j-2},$ and  $\mu$ itself. We compute the contribution of each of those weights, starting by determining all pairs $(\beta,r)\in I_{\nu_1}$ as in Theorem  \ref{How_to_determine_contributions}. A straightforward computation yields

$$
B_{\nu_1}=\tfrac{1}{2}(2n+1,2n-1,\ldots,\underbrace{2n-2j+1}_{(j+1)^\textnormal{th}\mbox{ } \textnormal{coord.}},2n-2j-3,\ldots ),
$$
and since $\varpi-\nu_1$ has support $\{\beta_1,\ldots,\beta_j\},$ we get that $\beta=\varepsilon_1-\varepsilon_{j+1}$ by definition of $I_{\nu_1}.$ Also, one easily checks that for $r\in \Z,$ we have 
$$
A_{\beta,r}=\tfrac{1}{2}(2n+3-2r,2n-1,\ldots,\underbrace{2n-2j+3}_{j^\textnormal{th}\mbox{ } \textnormal{coord.}},2n-2j-1+2r,2n-2j-3,\ldots,1).
$$ 
Therefore $B_{\nu_1}\in \mathscr{W}\cdot A_{\beta,r}$ if and only if $\{|2(n-r)+3|,|2(n-j+r)-1|\}=\{2n+1,2(n-j)+1\}.$ Again, we deal with each possibility separately.

\begin{enumerate}
\item If $2(n-r)+3=2n+1,$ then $ r=1$ and so $\nu_p(r)=0.$ Hence the weight $\varpi-\beta$ cannot contribute to $\nu^c_\mu(T_\varpi)$ in this situation.
\item If $2(n-r)+3=-2n-1,$ then $r=2(n+1),$ so that $ 2(n-j+r)-1  = 6n-2j+3.$ As it is impossible for the latter to be equal to $2(n-j)+1,$ we get the that $\varpi-2(n+1)\beta$ does not contribute to the truncated sum either.
\item If $2(n-r)+3=2(n-j)+1,$ then $r=j+1$ and $2(n-j+r)-1=2n+1,$ in which case  one checks that $B_\mu =s_{\varepsilon_1-\varepsilon_{j+1}}\cdot A_{\beta,j+1}.$ 
\item If $2(n-r)+3=-2(n-j)-1,$ then $r=2n-j+2 ,$ in which case $2(n-j+r)-1= 6n-4j+3 \neq 2n+1.$ Therefore $\varpi-(2n-j+2)\beta$ does not contribute in this case.
\end{enumerate} 

Consequently $I_{\nu_1}=\{(\varepsilon_1-\varepsilon_{j+1},j+1)\}$ and hence an application of Theorem \ref{How_to_determine_contributions} (where we take $w_{\varepsilon_1-\varepsilon_{j+1},j+1}= s_{\varepsilon_1-\varepsilon_j}$ yields the contribution $\nu_p(j+1)$ for $\nu_1$ to $\nu^c_\mu(T_\varpi).$
\vspace{5mm}

Next applying Lemma \ref{Dimension_of_various_irreducibles_in_type_B} to the $B_{n-j+2}$-Levi subgroup of $G$ corresponding to the simple roots $\beta_{j-1},\ldots,\beta_n,$ we get that $\nu_2$ does not afford the highest weight of a composition factor of $V_G(\varpi).$ Therefore $\nu_2$ cannot contribute to $\nu_\mu^c(T_\varpi)$ by Lemma \ref{no_contribution_for_certain_weights} and it only remains to consider the dominant $T$-weight $\mu.$ Here we have
$$
B_\mu=\tfrac{1}{2}(2n+1,2n-1,\ldots,\underbrace{2(n-j)+5,}_{(j-1)^\textnormal{th}\mbox{ } \textnormal{coord.}}2(n-j)+1, 2(n-j)-1,\ldots,3,1),
$$
and since $\varpi-\mu$ has support $\Pi,$ we get that $\beta\in \{\varepsilon_1,\varepsilon_1+\varepsilon_2,\varepsilon_1+\varepsilon_3,\ldots,\varepsilon_1+\varepsilon_n\}.$ Also for $r\in \Z,$ we have 
$$
A_{\varepsilon_1,r}=B_\mu +(1-r,\underbrace{0,\ldots,0}_{j-2 },1,0,\ldots,0).
$$
Hence $|(A_{\varepsilon_1,r})_j|$ is distinct from $(B_{\mu})_t$ for all $1\leqslant t\leqslant n,$ showing that  $\mu\notin \mathscr{W}\cdot (\varpi-\Z\varepsilon_1).$ Arguing in a similar fashion, one checks that if $\beta= \varepsilon_1+\varepsilon_\ell$ for some $2\leqslant \ell\leqslant n$ different from $j,$ then $\mu\notin \mathscr{W}\cdot (\varpi-\beta\mathbb{Z} ).$ Finally,   consider $\beta=\varepsilon_1+\varepsilon_j,$ in which case
$$
A_{\beta,r}=B_\mu +(1-r,\underbrace{0,\ldots,0}_{j-2 },1-r,0,\ldots,0),~r\in \Z.
$$
Therefore  $B_\mu \in \mathscr{W} \cdot A_{\beta,r}$ if and only if $\{|2(n-r)+3|,|2(n-j-r)+3|\}=\{2n+1,2(n-j)+1\}.$ Now clearly $2(n-r)+3\neq 2n+1,$ for it would force $r=1.$ Also if $2(n-r)+3=-2n-1,$ then $r=2(n+1),$ in which case $|2(n-j-r)+3|= 2(n+j)+1\neq  2(n-j)+1 .$ If on the other hand $2(n-r)+3=2(n-j)+1,$ then $r=j+1,$ and again one checks that $|2(n-j-r)+3| \neq 2n+1$ in this situation. Finally, arguing in a similar fashion shows that $ B_\mu \in \mathscr{W}\cdot A_{\beta,r}$ if and only if $r=2n-j+2,$ in which case $\mu=(s_{\varepsilon_1} s_{\varepsilon_j} s_{\varepsilon_1-\varepsilon_j}) \cdot (\varpi-r\beta).$ Consequently 
$$
\nu_{\mu}^c(T_\varpi)=\nu_p(j+1)\chi_\mu(\varpi_{j+1} + \delta_{j,n-1}\varpi_n) + \nu_p(2n-j+2)\chi_\mu(\mu).
$$
In order to conclude, first observe that $\chi_\mu(\mu)=\ch L_G(\mu)$ by definition, while Lemma \ref{Dimension_of_various_irreducibles_in_type_B} yields $\chi_\mu(\varpi_{j+1} + \delta_{j,n-1}\varpi_n)=\ch L_G(\varpi_{j+1} + \delta_{j,n-1}\varpi_n).$ The proof is complete.
\end{proof}
%---------------------------------------------------------%
%---------------------------------------------------------%

Proceeding as in the proofs of Proposition \ref{Type_B:Sum_formula_for_(2,0,...,0)}, we next compute the formal character $\nu^c_\mu(T_\varpi)$ in the situation where $\varpi=\varpi_1+\varpi_n$  and $\mu=\varpi_{n}.$

%---------------------------------------------------------%
%-----     Type B, sum formula for (1,0,...,0,1)     -----%
%---------------------------------------------------------%
\begin{prop}\label{Type_B:Sum_formula_for_(1,0,...,0,1)} 
Assume $\varpi= \varpi_1 + \varpi_n,$ and write $\mu=\varpi_n.$ Let $V_G(\varpi)=V^0\supsetneq V^1\supseteq \ldots \supseteq V^k \supseteq 0$ be a Jantzen filtration of $V_G(\varpi).$  Then 
$$
\nu^c_\mu(T_\varpi)=\nu_p(2n+1) \ch L_G(\mu).
$$
\end{prop}
 
\begin{proof}
We proceed as in the proof of Propositions \ref{Type_B:Sum_formula_for_(2,0,...,0)} and \ref{Type_B:Sum_formula_for_(1,0,...,0,1,0,...,0)}. Here  $\Lambda^+(\varpi) = \{\varpi,\mu\},$ thus we only need to   compute the contribution of  all pairs $(\beta,r)\in I_{\mu}$ as in Theorem  \ref{How_to_determine_contributions}. We have
$$
B_\mu=(n,n-1,\ldots,2,1),
$$
and since $\varpi-\mu=\beta_1+\cdots+\beta_n$ has support $\Pi,$ we get that  $\beta\in \{\varepsilon_1,\varepsilon_1+\varepsilon_2,\varepsilon_1+\varepsilon_3,\ldots,\varepsilon_1+\varepsilon_n\} $ by definition of $I_\mu.$  Let then $2\leqslant \ell\leqslant n$ be fixed and assume $\beta=\varepsilon_1+\varepsilon_\ell,$ in which case one checks that we have 
$$
A_{\beta,r}= (n+1-r,n-1,n-2,\ldots,2,1) + (\underbrace{0,\ldots,0}_{\ell-1},-r,0,\ldots,0),~r \in \Z.
$$
Consequenly  $B_\mu \in \mathscr{W}\cdot A_{\beta,r}$ if and only if $\{|n+1-r|,|n+1-\ell-r|\}=\{n,n+1-\ell\}.$ Proceeding exactly as in the proofs of Propositions \ref{Type_B:Sum_formula_for_(2,0,...,0)} and \ref{Type_B:Sum_formula_for_(1,0,...,0,1,0,...,0)}, one easily shows that the latter set equality never holds. Therefore $\mu\notin \mathscr{W}\cdot (\varpi-r\beta)$ for $r\in \Z$  in this situation. Finally, we also  leave to the reader to check that $ B_\mu \in \mathscr{W}\cdot  A_{\varepsilon_1,r}$ if and only if $r=2n+1,$ in which case $\mu=s_{\varepsilon_1}\cdot (\varpi-r\varepsilon_1).$ Hence $\nu^c(T_\varpi)=\nu_p(2n+1)\chi(\mu),$ and since $V_G(\mu)$ is irreducible by Lemma \ref{Dimension_of_various_irreducibles_in_type_B}, we get that $\chi(\mu)=\ch L_G(\mu),$ from which the result follows.
\end{proof}
%---------------------------------------------------------%
%---------------------------------------------------------%

We conclude this section by computing the character $\nu^c_\mu(T_\varpi)$ in the situation where $\varpi=\varpi_1+2\varpi_n$  and $\mu=\varpi_{n-1}.$ Again, we proceed as in the proofs of Propositions \ref{Type_B:Sum_formula_for_(2,0,...,0)}, \ref{Type_B:Sum_formula_for_(1,0,...,0,1,0,...,0)}, and \ref{Type_B:Sum_formula_for_(1,0,...,0,1)}.

%---------------------------------------------------------%
%-----     Type B, sum formula for (1,0,...,0,2)     -----%
%---------------------------------------------------------%
\begin{prop}\label{Type_B:Sum_formula_for_(1,0,...,0,2)} 
Assume $\varpi= \varpi_1 + 2\varpi_n,$ and write $\mu=\varpi_{n-1}.$ Let $V_G(\varpi)=V^0\supsetneq V^1\supseteq \ldots \supseteq V^k \supseteq 0$ be a Jantzen filtration of $V_G(\varpi).$  Then 
$$
\nu^c_{\mu}(T_\varpi)= \nu_p(n+1)\ch L_G(2\varpi_n) + \nu_p( n+2) \ch L_G(\mu).
$$
\end{prop}

\begin{proof}
We proceed as usual, starting by observing that the dominant weights $\nu\in X^+(T)$ satisfying $\mu\preccurlyeq \nu \prec \varpi$ are  $\varpi-\beta_n,$ and  $\nu_1=2\varpi_n.$ Since $\varpi-\beta_n$ has multiplicity $1$ in $V_G(\varpi),$ it cannot afford the highest weight of a composition factor of $V_G(\varpi)$ by \cite{Pr}, and hence does not contribute to $\nu_{\mu}^c(T_\varpi)$ by Lemma \ref{no_contribution_for_certain_weights}. We  next compute the contribution of all pairs $(\beta,r)\in I_{\nu_1}$ as in Theorem  \ref{How_to_determine_contributions}. Here 
$$
B_{\nu_1}=\tfrac{1}{2}(2n+1,2n-1,\ldots,3),
$$
and since $\varpi-\nu_1=\beta_1+\cdots+\beta_n$ has support $\Pi,$ we get that  $\beta\in \{\varepsilon_1,\varepsilon_1+\varepsilon_2,\varepsilon_1+\varepsilon_3,\ldots,\varepsilon_1+\varepsilon_n\} $ by definition of $I_{\nu_1}.$  Now for any pair $(\beta,r)$ in $I_{\nu_1},$ we have 
$$
A_{\beta,r}= \tfrac{1}{2}(2n+3,2n-1,2n-3,\ldots,3)-r\beta.
$$
One checks that $B_{\nu_1}$ and $A_{\beta,r}$ are $\mathscr{W}$-conjugate for the dot action if and only if $\beta=\varepsilon_1$ and $r=2(n+1),$ in which case $\nu_1=s_{\varepsilon_1}\cdot (\varpi-2(n+1)\beta).$ Hence the $T$-weight $\nu_1$ contributes by $ \nu_p(n+1)$ to $\nu_\mu^c(T_\varpi)$ by Theorem \ref{How_to_determine_contributions}. We  now compute the contribution of all pairs $(\beta,r)\in I_{\mu}$ as in Theorem  \ref{How_to_determine_contributions}. Here we have
$$
B_{\mu}=\tfrac{1}{2}(2n+1,2n-1,\ldots,5,1).
$$
Again, since $\varpi-\mu$ has support $\Pi,$ we get that  $\beta\in \{\varepsilon_1,\varepsilon_1+\varepsilon_2,\varepsilon_1+\varepsilon_3,\ldots,\varepsilon_1+\varepsilon_n\} $ by definition of $I_{\mu}.$ 
Now for any pair $(\beta,r)$ in $I_{\mu},$ we have 
$$
A_{\beta,r}= \tfrac{1}{2}(2n+3,2n-1,2n-3,\ldots,3)-r\beta,
$$
and hence if $B_{\mu}$ and $A_{\beta,r}$ are $\mathscr{W}$-conjugate for the dot action, then $\beta=\varepsilon_1+\varepsilon_n,$ which we shall assume holds in what follows. One then checks that $B_{\mu} \in \mathscr{W}\cdot A_{\beta,r}$ if and only if $r=n+2,$ and that $\mu=(s_{\varepsilon_1}s_{\varepsilon_n}s_{\varepsilon_1-\varepsilon_n})\cdot (\varpi-(n+2)\beta)$ in this situation. In particular, we get that $\mu$ contributes by $\nu_p(n+2) $ to $\nu^c_{\mu}(T_\varpi)$ by Theorem  \ref{How_to_determine_contributions}.  Therefore $\nu^c_{\mu}(T_\varpi)=\nu_p(n+1)\chi_{\mu}(\nu_1) + \nu_p(n+2)\chi_{\mu}(\mu),$ and since $\m_{V_G(\nu_1)}(\mu)=1,$ an application of \cite{Pr} yields $\chi_{\mu}(\nu_1)=\ch L_G(\nu_1).$ Also we clearly have $\chi_{\mu}(\mu)=\ch L_G(\mu),$ thus completing the proof.
\end{proof}
%---------------------------------------------------------%
%---------------------------------------------------------%

%%%%%%%%%%%%%%%%%%%%%%%%%%%%%%%%%
\subsection{Conclusion}     %%%%%
%%%%%%%%%%%%%%%%%%%%%%%%%%%%%%%%%

We conclude this section by giving proofs of Theorem \ref{Main_Result_1},  Corollary  \ref{Introduction: corollary on the dimension (B)}, and Proposition \ref{Restrictions_of_irreducibles_of_type_A}, in the situation where $G$ is of type $B_n$ $(n\geqslant 2)$ and $Y$ is of type $A_{2n}.$

%----------------------------------------------%
%------ Proof of Theorem 1 in the B-case ------%
%----------------------------------------------%
\begin{proof}[Proof of Theorem \textnormal{\ref{Main_Result_1}} (The $B_n$-case)]
We deal with each row of Table \ref{Table_B} separately, starting with the case where $\varpi=2\varpi_1.$ Here by Lemma \ref{chT(omega)_for_B}, we have $\ch T(\varpi) =\chi(\varpi)+\ch L_G(0),$  and an application of Proposition \ref{chT(omega)_as_a_sum_of_xi's} yields $[V_G(\varpi),L_G(\mu)]=0$ for   $\mu\in X^+(T)$ different from $\varpi$ or $0.$ Moreover, the same proposition also shows that if $0$ affords the highest weight of a composition factor of $V_G(\varpi),$ then $[V_G(\varpi),L_G(0)]=1.$ Now 
$$
\nu_0^c(T_\varpi)=\nu_p(2n+1)\ch L_G(0)
$$
by Proposition \ref{Type_B:Sum_formula_for_(2,0,...,0)}, and so the zero weight  affords the highest weight of a composition factor if and only if $p$ divides $ 2n+1$ by Proposition \ref{Insight_on_possible_composition_factors}. The desired assertion thus holds in this situation.

We next consider the case where $\varpi=\varpi_1+\varpi_j$ for some $1<j<n.$ Here again, an application of Lemma \ref{chT(omega)_for_B}   yields $\ch T(\varpi) = \chi(\varpi) + \chi(\varpi_{j-1})+ \chi(\varpi_{j+1} + \delta_{j,n-1}\varpi_n),$ so that $[V_G(\varpi),L_G(\mu)]=0$ for  $\mu\in X^+(T)$ different from $\varpi,$ $\varpi_{j-1},$ or $\varpi_{j+1}+\delta_{j,n-1}\varpi_n,$ by Proposition  \ref{chT(omega)_as_a_sum_of_xi's}. The latter also shows that if  $\mu\in \{\varpi_{j-1},\varpi_{j+1}+\delta_{j,n-1}\varpi_n\}$ affords the highest weight of a composition factor of $V_G(\varpi),$ then $[V_G(\varpi),L_G(\mu)]=1.$ Applying Proposition \ref{Type_B:Sum_formula_for_(1,0,...,0,1,0,...,0)} yields 
$$
\nu_{\varpi_{j-1}}^c(T_\varpi)= \nu_p(j+1)\ch L_G(\varpi_{j+1}+\delta_{j,n-1}\varpi_n) + \nu_p(2n-j+2)\ch L_G(\varpi_{j-1}).
$$
Arguing as in the previous case, one gets that  $V_G(\varpi)=\varpi/(\varpi_{j+1}+\delta_{j,n-1}\varpi_n)^{\epsilon_p(j+1)}/ \varpi_{j-1}^{\epsilon_p(2n-j+2)}$ by Proposition \ref{Insight_on_possible_composition_factors}. Finally,  both   $V_G(\varpi_{j+1}+\delta_{j,n-1}\varpi_n)$ and $V_G(\varpi_{j-1})$ are irreducible by Lemma \ref{Dimension_of_various_irreducibles_in_type_B}, and hence $\Ext^1_G(L_G(\varpi_{j-1}),L_G(\varpi_{j+1}+\delta_{j,n-1}\varpi_n))=0$  by Proposition \ref{Extensions between two irreducibles}. The desired result thus holds in this case as well.
\vspace{5mm}

Next consider the dominant $T$-weight $\varpi=\varpi_1+\varpi_n.$ Here $\ch T(\varpi)=\chi(\varpi)+\chi(\varpi_n) $ by Lemma \ref{chT(omega)_for_B}, so that  $[V_G(\varpi),L_G(\mu)]=0$ for $\mu\in X^+(T)$ different from $\varpi$ or $\varpi_n$ by Proposition \ref{chT(omega)_as_a_sum_of_xi's}. The latter also shows that if  $\varpi_n$ affords the highest weight of a composition factor of $V_G(\varpi),$ then $[V_G(\varpi),L_G(\varpi_n)]=1.$ Applying Proposition \ref{Type_B:Sum_formula_for_(1,0,...,0,1)} yields 
$$
\nu_{\varpi_n}^c(T_\varpi)= \nu_p(2n+1)\ch L_G(\varpi_n),
$$
and so $\varpi_n$ affords the highest weight of a composition factor of $V_G(\varpi)$ if and only if $p$ divides $2n+1$ by Proposition \ref{Insight_on_possible_composition_factors}. Hence the assertion holds as desired.  
\vspace{5mm}

Finally, we consider the situation where $\varpi=\varpi_1+2\varpi_n.$ Here an application of Lemma \ref{Another_chT(omega)_for_B} yields $\ch T(\varpi)=\chi(\varpi)+ \chi(\varpi_{n-1})+\chi(2\varpi_n),$ while $[V_G(\varpi),L_G(\mu)]=0$ for $\mu\in X^+(T)$ different from $\varpi,$ $\varpi_{n-1},$ or $2\varpi_{n} $  by Proposition \ref{Type_B:Sum_formula_for_(1,0,...,0,1,0,...,0)}. The latter also shows that if  $\mu=2\varpi_n$ affords the highest weight of a composition factor of $V_G(\varpi),$ then $[V_G(\varpi),L_G(\mu)]=1.$ Applying Proposition \ref{Type_B:Sum_formula_for_(1,0,...,0,2)}  yields 
$$
\nu^c_{2\varpi_n}(T_\varpi)= \nu_p(n+1)\ch L_G(\varpi_{n-1}) + \nu_p( n+2) \ch L_G(2\varpi_n).
$$
Using Proposition \ref{Insight_on_possible_composition_factors}, we then get that $V_G(\varpi)=\varpi/(\varpi_{j+1}+\delta_{j,n-1}\varpi_n)^{\epsilon_p(j+1)}/ \varpi_{j-1}^{\epsilon_p(2n-j+2)}.$ Here again, a suitable application of Proposition \ref{Extensions between two irreducibles} allows  us to conclude, thus completing the proof.
\end{proof}
%----------------------------------------------%
%----------------------------------------------% 

The proof of Corollary \ref{Introduction: corollary on the dimension (B)} in the $B_n$-case simply consists in computing the dimension of $V_G(\varpi)$ and $\rad(\varpi),$ this for each $\varpi$ appearing in the first column of Table \ref{Table_B}. The details are thus left to the reader. We conclude with a proof of Proposition \ref{Restrictions_of_irreducibles_of_type_A} in the case where $G$ is of type $B_n.$

%------------------------------------------------%
%----- Proof of Proposition 3 in the B-case -----%
%------------------------------------------------%
\begin{proof}[Proof of Proposition \textnormal{\ref{Restrictions_of_irreducibles_of_type_A}} (The $B_n$-case)]
We give a proof in the situation where $1\leqslant j\leqslant n-1,$ and omit  the other cases, as they can be  dealt with in a similar fashion. First observe that by Lemma \ref{Tensor_products_in_type_A}, we have
$$
\ch V|_G=\ch V_Y(\lambda)|_G-\epsilon_p(j+1)\ch L_Y(\lambda_{j+1})|_G.
$$
Also, an application of Proposition \ref{Type_B:restriction_in_characteristic_zero_0<j<n+1} yields $\ch V_Y(\lambda)|_G = \chi(\varpi)+\ch L_G(\varpi_{j-1}),$ while by \cite[Theorem 1, Table 1 $(\textnormal{I}_2,\textnormal{I}_3)$]{Se}, we have $\ch L_Y(\lambda_{j+1})|_G= \ch L_G(\varpi_{j+1}+\delta_{j,n-1}\varpi_n).$ Therefore
$$
\ch V|_G=\chi(\varpi)+ \ch L_G(\varpi_{j-1}) -\epsilon_p(j+1)\ch L_G(\varpi_{j+1}+\delta_{j,n-1}\varpi_n).
$$
Now $\chi(\varpi)= \ch L_G(\varpi)+\epsilon_p(j+1)\ch L_G(\varpi_{j+1}+ \delta_{j,n-1}\varpi_n) + \epsilon_p(2n-j+2)\ch L_G(\varpi_{j-1}) $ by  Theorem \ref{Main_Result_1}, and so the result follows.
\end{proof}
%------------------------------------------------%
%------------------------------------------------% 

%%%%%%%%%%%%%%%%%%%%%%%%%%%%%%%%%%%%%%%%%%%%%%
%%%%%%%%%%%%%%%%%%%%%%%%%%%%%%%%%%%%%%%%%%%%%% 
\section{The $D_n$-case $(n\geqslant 4)$}		%%%%%
%%%%%%%%%%%%%%%%%%%%%%%%%%%%%%%%%%%%%%%%%%%%%%
%%%%%%%%%%%%%%%%%%%%%%%%%%%%%%%%%%%%%%%%%%%%%%

Let $K$ be an algebraically closed field of characteristic $p\neq 2,$ and let $Y$  a simply connected simple algebraic group of type $A_{2n-1}$ $(n\geqslant 3)$ over $K.$  Consider a subgroup $G$ of type $D_n,$ embedded in the usual way, as the stabilizer of a non-degenerate quadratic form on the natural module for $Y.$ Fix a Borel subgroup $B_Y = U_Y T_Y$ of $Y,$ where $T_Y$ is a maximal torus of $Y$ and $U_Y$ is the unipotent radical of $B_Y,$ let $\Pi(Y ) = \{\alpha_1,\ldots,\alpha_{2n-1}\}$ denote a corresponding base of the root system $\Phi(Y)=\Phi^+(Y)\sqcup \Phi^-(Y)$ of $Y,$ and let $\{\lambda_1,\ldots,\lambda_{2n-1}\}$ be the set of fundamental dominant weights for $T_Y$ corresponding to our choice of base $\Pi(Y).$ Also set $T=T_Y\cap G,$ $B=B_Y\cap G,$ so that $T$ is a maximal torus of $G,$ and $B$ is a Borel subgroup of $G$ containing $T.$ Let $\Pi(G)=\{\beta_1,\ldots,\beta_n\}$ be the corresponding base for the root system $\Phi(G)=\Phi^+(G)\sqcup \Phi^-(G)$ of $G,$ and let $\varpi_1,\ldots,\varpi_n$ denote the associated fundamental weights. The $A_{n-1}$-parabolic subgroup of $G$ corresponding to the simple roots $\{\beta_1,\ldots,\beta_{n-1}\}$ embeds in an $A_{n-1}\times A_{n-1}$-parabolic subgroup of $Y,$ and up to conjugacy, we may assume that this gives $\alpha_i|_{T}=\alpha_{2n-1-i}|_{T}=\beta_i$ for $1\leqslant i\leqslant n-1.$ By considering the action of the Levi factors of these parabolics on the natural $KY$-module $L_Y(\lambda_1),$ we can deduce that $\alpha_n|_{T}=\beta_n-\beta_{n-1}.$ Finally, using \cite[Table 1, p.69]{Humphreys1} and the fact that $\lambda_1|_{T}=\varpi_1$ yields
\begin{equation}
\lambda_i|_{T}=\lambda_{2n-i}|_{T}=\varpi_i, \mbox{ } \lambda_{n-1}|_{T}=\lambda_{n+1}|_{T}= \varpi_{n-1}+\varpi_n, \mbox{ } \lambda_n|_{T}=2\varpi_n \mbox{ for $1\leqslant i\leqslant n-2.$}
\label{[Dn<A2n-1] Weight restrictions for D}
\end{equation}

In this section, we show how to obtain the tables \ref{Table_D}, \ref{Table_of_dimensions_main_result_type_D},  and \ref{Restriction_of_V_Y(lambda)_to_G_of_type_D} in Theorem \ref{Main_Result_1}, Corollary \ref{Introduction: corollary on the dimension (B)}, and Proposition \ref{Restrictions_of_irreducibles_of_type_A}, respectively. We proceed as in Section \ref{The_Bn_case}, relying as much as possible on the embedding of $G$ in $Y$ described above.

%%%%%%%%%%%%%%%%%%%%%%%%%%%%%%%%%%%%%%%%%%%%%%%%%%%%%%%%%%%%%%%%%%%%%%%%%
\subsection{Restriction to $G$ of certain Weyl modules for $Y$}     %%%%%
%%%%%%%%%%%%%%%%%%%%%%%%%%%%%%%%%%%%%%%%%%%%%%%%%%%%%%%%%%%%%%%%%%%%%%%%%

Let $V$ be a $KY$-module with unique  highest weight $\lambda\in X^+(T_Y).$ We start by proving a result similar to Lemma \ref{[B]restriction_of_highest_weight}.  Contrary to what we had in the latter result, the $T$-weight $\lambda|_T$ is not necessarily the unique highest weight of $V|_G.$ (For example, if $\langle \lambda,\alpha_n\rangle \neq 0,$ then the weight $\lambda-\alpha_n$ restricts to $\lambda|_T+\beta_{n-1}-\beta_n,$ which is neither under nor above $\lambda|_T.)$ Notice that according to the restrictions to $T$ of the simple roots for $T_Y$ stated above, we have that the restriction to $T$ of a given $T_Y$-weight $\mu=\lambda- \sum_{r=1}^{2n-1}{c_r \alpha_r}$ is given by
\begin{equation}
\mu|_T=\lambda|_T-\sum_{r=1}^{n-2}{(c_r+c_{2n-r})\beta_r -(c_{n-1}-c_n+c_{n+1})\beta_{n-1} - c_n\beta_n}.
\label{restriction_of_any_TY-weight_type_D}
\end{equation}

%-------------------------------------------------------------------%
%----- If <lambda,alphan>=0 all weights in V|X are under omega -----%
%-------------------------------------------------------------------%
\begin{lem}\label{[D]restriction_of_highest_weight1}
Let $\lambda\in X^+(T_Y)$ be a dominant weight, and let $V$ be a $KY$-module with unique highest weight $\lambda.$ Then the following assertions hold.
\begin{enumerate}
\item \label{restriction_of_weights_1} If $\langle \lambda, \alpha_n\rangle =0,$ then every $T$-weight $\xi$ of $V$ satisfies $\xi \preccurlyeq \lambda|_T.$
\item \label{restriction_of_weights_2} If $\langle \lambda, \alpha_n\rangle =1,$ then each of $\lambda|_{T}$ and $(\lambda -\alpha_n)|_{T}$ affords the highest weight of a $KG$-composition factor of $V.$ Furthermore, every $T$-weight $\xi$ of $V$ either satisfies $\xi\preccurlyeq \lambda|_T$ or $\xi \preccurlyeq (\lambda-\alpha_n)|_T.$
\end{enumerate}
\end{lem}

\begin{proof}
Write $\varpi=\lambda|_{T}$ and $\varpi'=(\lambda-\alpha_n)|_T.$ First suppose that $\langle \lambda,\alpha_n\rangle=0$ and, seeking a contradiction, assume the existence of non-negative integers $c_1,\ldots,c_{2n-1}\in \mathbb{Z}_{\geqslant 0}$ such that $\mu=\lambda- \sum_{r=1}^{2n-1}{c_r \alpha_r}$  satisfies $\xi=\mu |_{T}\not\preccurlyeq \varpi.$ By \eqref{restriction_of_any_TY-weight_type_D}, we get that $c_n>c_{n-1}+c_{n+1}.$ In particular, we get that $\langle \mu, \alpha_n\rangle < -c_n,$ from which one deduces that $s_{\alpha_n}(\mu)$ is a $T_Y$-weight of $V$ that is not under $\lambda,$ a contradiction. Therefore \ref{restriction_of_weights_1} holds as desired, and we assume $\langle \lambda, \alpha_n\rangle =1$ in the remainder of the proof. 

We first show that $\varpi$ and $\varpi'$ are  highest weights of $V|_G,$ and hence each of them affords the highest weight of a $KG$-composition factor of $V.$ Let $\mu=\lambda-\sum_{r=1}^{2n-1}{c_r\alpha_r}$ be a $T_Y$-weight of $V$   such that $\varpi'\preccurlyeq \mu|_{T}.$ By \eqref{restriction_of_any_TY-weight_type_D}, we get $c_r=0$ for every $1\leqslant r\leqslant 2n-1$ different from $n-1,n,n+1,$ as well as $c_n=1,$ and hence $c_{n-1}+c_{n+1}=0.$ Therefore $ c_{n-1}=c_{n+1}=0$ and thus $\mu|_{T}=\varpi'$ as desired. In order to prove the last assertion, assume for a contradiction the existence of a $T_Y$-weight $\mu=\lambda-\sum_{r=1}^{2n-1}{c_r\alpha_r}$ of $V$ such that neither $\mu|_{T}\preccurlyeq \varpi$ nor $\mu|_{T}\preccurlyeq \varpi'.$ By \eqref{restriction_of_any_TY-weight_type_D} again, we have $c_{n-1}-c_n+c_{n+1}<-1.$ In particular $\langle \mu , \alpha_n \rangle <-c_n,$ showing that $s_{\alpha_n}(\mu)$ is a $T_Y$-weight of $V$ that is not under $\lambda,$ a contradiction.
\end{proof}
%-------------------------------------------------------------------%
%-------------------------------------------------------------------%

%----------------------------------------%
%------ Corollary on V_G(\lambda_n) -----%
%----------------------------------------%
\begin{corol}\label{Corollary_on_V(lambda_n)}
Let $\lambda=\lambda_n\in X^+(T_Y),$ and consider the Weyl module $V_Y(\lambda)$ having highest weight $\lambda.$ Then the formal character of the restriction $V_Y(\lambda)$ to $G$ is given by
$$
\ch V_G(\lambda)|_G = \chi(2\varpi_{n-1}) + \chi(2\varpi_n).
$$ 
\end{corol}

\begin{proof}
As usual, we assume $K$ has characteristic zero. Now by part \ref{restriction_of_weights_2} of Lemma \ref{[D]restriction_of_highest_weight1}, each of $\varpi=2\varpi_n$ and $\varpi'=2\varpi_{n-1}$ affords the highest weight of a composition factor of $V_Y(\lambda)|_G.$  An application of Tables \ref{Various_dimensions_in_type_A} and \ref{Various_dimensions_in_type_D} then yields $\dim V_Y(\lambda) = \dim V_G(2\varpi_{n-1})+\dim V_G(2\varpi_n),$ thus completing the proof.
\end{proof}
%----------------------------------------%
%----------------------------------------%

We next investigate the formal character of the restriction to $G$ of the Weyl module $V_Y(\lambda_1+\lambda_j),$ where $1\leqslant j \leqslant 2n-1.$

%----------------------------------------------------------\
%------     ch V(lambda_1 + lambda_j) for n<j<2n     ------\
%----------------------------------------------------------\
\begin{prop}\label{[D]character_of_V_Y(lambda)_G_for_D}
Let $1\leqslant j\leqslant 2n-1$ and  consider the dominant $T_Y$-weight $\lambda=\lambda_1+\lambda_j\in X^+(T_Y).$ Also set $\varpi=\lambda|_T$  and adopt the notation $\varpi_0=0.$ Then  
$$
\ch V_Y(\lambda)|_G  =
\begin{cases}
\chi(\varpi) + \chi(\varpi_{j-1})	& \mbox{ if $1\leqslant j\leqslant n-1;$}\cr
\chi(\varpi)+\chi(\varpi_1+2\varpi_{n-1})+ \chi(\varpi_{n-1}+\varpi_n) & \mbox{ if $j=n;$}\cr 
\chi(\varpi) + \chi(2\varpi_{n-1}) + \chi(2\varpi_n) 													& \mbox{ if $j=n+1;$} 				\cr
\chi(\varpi) + \chi(\varpi_{n-1}+\varpi_n)										& \mbox{ if $j=n+2;$}  						\cr
\chi(\varpi) + \chi(\varpi_{2n-j+1})										& \mbox{ otherwise.}
\end{cases}
$$
\end{prop}
\begin{proof}
Write $V=V_Y(\lambda)$ and first observe that $\ch V|_G$ is independent of $p,$ so we may and shall assume $K$ has characteristic zero in the remainder of the proof.  In the case where $1\leqslant j\leqslant n-1,$ an application of \cite[Proposition 1.5.3]{Koike} yields $V_Y(\lambda)|_G\cong V_G(\varpi)\oplus V_G(\varpi_{j-1}),$ from which the result follows. 
\vspace{5mm}

Next assume $j=n,$ and write $\varpi'=\varpi_1+2\varpi_{n-1},$ $\varpi''=2\varpi_{n-1}+\varpi_n.$ By Part \ref{restriction_of_weights_2} of Lemma \ref{[D]restriction_of_highest_weight1}, each of $\varpi$ and $\varpi'$ affords the highest weight of a composition factor of $V|_G.$ Observe that the only $T_Y$-weights restricting to $\varpi''=\varpi_{n-1}+\varpi_n$ are $\lambda-(\alpha_1+\cdots+\alpha_r+\alpha_n+\cdots+\alpha_{2n-r-1})$ (for $1\leqslant r \leqslant n-1$) and $\lambda-(\alpha_n+\cdots+\alpha_{2n-1}).$ An application of Lemma \ref{Tensor_products_in_type_A} then yields $\m_{V|_G}(\varpi'') = 2n-1$ as well as $\m_{V_G(\varpi)}(\varpi'')=\m_{V_G(\varpi')}(\varpi'')=n-1,$ thus showing that $\varpi''$ occurs in a third $KG$-composition factor of $V.$  Now one easily checks that every $T$-weight $\nu\in \Lambda^+(V|_G)$ such that $\varpi''\prec\nu \prec \varpi$ or $\varpi''\prec \nu \prec \varpi'$ satisfies $
\m_{V|_G}(\nu)=\m_{V_G(\varpi)}(\nu)+\m_{V_G(\varpi')}(\nu),$
showing that $\varpi''$ affords the highest weight of a third $KG$-composition factor of $V$ by Part \ref{restriction_of_weights_2} of Lemma \ref{[D]restriction_of_highest_weight1}. As in the proof of Proposition \ref{Type_B:restriction_in_characteristic_zero_0<j<n+1}, an application of Theorem \ref{Weyl's formula} then yields the desired result.
\vspace{5mm}

Next suppose that $j=n+1$ and consider the dominant $T$-weight $\varpi'=\varpi-(\beta_1+\cdots+\beta_{n-1})\in X^+(T).$ Then the $T_Y$-weights restricting to $\varpi'$ are $\lambda-(\alpha_1+\cdots+\alpha_{n-1}),$ $\lambda-(\alpha_1+\cdots+\alpha_r+\alpha_{n+1}+\cdots+\alpha_{2n-r-1})$ $(1\leqslant r<n-1),$ and $\lambda-(\alpha_{n+1}+\cdots+\alpha_{2n-1}).$ Therefore $\m_{V|_G}(\varpi')=n,$ while $\m_{V_G(\varpi)}(\varpi')=n-1$ by Lemma \ref{Tensor_products_in_type_A}, showing that $\varpi'$ occurs in a second $KG$-composition factor of $V.$ Since there is no dominant weight $\nu \in X^+(\varpi)$ such that $\varpi'\prec \nu \prec \varpi,$ an application of Lemma \ref{[D]restriction_of_highest_weight1} yields $[V|_G,L_G(\varpi')]=1$ as desired. We leave to the reader to show that $[V|_G,L_G(2\varpi_{n-1})]=1$ as well, from which one easily concludes thanks to Theorem \ref{Weyl's formula}, for example.
\vspace{5mm}

Next consider $n+1<j \leqslant 2n-1$ and let  $\varpi'=\varpi-(\beta_1+\cdots+\beta_{2n-j}) \in X^+(T).$ Then one easily checks that the $T_Y$-weights restricting to $\varpi'$  are $\lambda-(\alpha_j+\cdots+\alpha_{2n-1}),$ $\lambda-(\alpha_1+\cdots+\alpha_r + \alpha_j+\cdots+\alpha_{2n-r-1})$ $(1\leqslant r \leqslant 2n-j-1),$ and $\lambda-(\alpha_1+\cdots+\alpha_{2n-j}).$ Therefore $\m_{V|_G}(\varpi')=2n-j+1,$ while an application of Lemma \ref{Tensor_products_in_type_A} yields $\m_{V_G(\varpi)}(\varpi')=2n-j,$ showing that $\varpi'$ occurs in a second $KG$-composition factor of $V.$ As above, there is no dominant weight $\nu \in X^+(\varpi)$ such that $\varpi'\prec \nu \prec \varpi$ and thus $[V|_G,L_G(\varpi')]=1$ by Lemma \ref{[D]restriction_of_highest_weight1}. Again, applying Theorem \ref{Weyl's formula} completes the proof.
\end{proof}
%----------------------------------------------------------\
%----------------------------------------------------------\

%%%%%%%%%%%%%%%%%%%%%%%%%%%%%%%%%%%%%%%%%%%%%%%%%%%%%%%%%%%%%%%%%%
\subsection{Formal character of certain tensor products}     %%%%%
%%%%%%%%%%%%%%%%%%%%%%%%%%%%%%%%%%%%%%%%%%%%%%%%%%%%%%%%%%%%%%%%%%

We next determine the formal character of the tensor product $V_G(\varpi_1)\otimes V_G(\varpi_j+\delta_{j,n}\varpi_n)$ for $1\leqslant j\leqslant n,$ as well as of the formal character of the tensor product $V_G(\varpi_1)\otimes V_G(\varpi_{n-1}+\varpi_n).$ Observe that since $p\neq 2,$ each of the Weyl modules considered is irreducible, and hence is tilting.

%-------------------------------------%
%-----     chT(\varpi) for D     -----%
%-------------------------------------%
\begin{lem}\label{chT(omega)_for_D}
Let $1\leqslant j\leqslant n-1,$ and consider the dominant $T$-weight $\varpi=\varpi_1+\varpi_j + \delta_{j,n-1}\varpi_n\in X^+(T).$ Also set $\varpi_0=\varpi_{n+1}=0$ and write $T(w)$ for the tensor product $V_G(\varpi_1)\otimes V_G(\varpi_j+\delta_{j,n-1}\varpi_n).$ Then $T(\varpi)$ is tilting and its formal character is given by 
$$
\ch T(\varpi)=\chi(\varpi) + \chi(\varpi_{j-1}) +  (1-\delta_{j,n-1})\chi(\varpi_{j+1}+\delta_{j,n-2}\varpi_n) + \delta_{j,n-1}(\chi(2\varpi_{n-1})+\chi(2\varpi_n)).
$$
\end{lem}

\begin{proof}
By Lemma \ref{Dimension_of_various_irreducibles_in_type_D}, both $V_G(\varpi_1)$ and $V_G(\varpi_j+\delta_{j,n-1}\varpi_n)$ are irreducible $KG$-modules, and hence $T(\varpi)$ is tilting by Proposition \ref{Tensor product of modules with good filtrations}. Also $\ch T(\varpi)$ is independent of $p,$ so we may and shall assume $K$ has characteristic zero in the remainder of the proof. By \cite[Theorem 1, Table 1 $(\textnormal{I}_4,\textnormal{I}_5)$]{Se} together with Lemma \ref{Tensor_products_in_type_A}, we successively get 
\begin{equation*}
T(\varpi)	\cong	\left(V_Y(\lambda_1)\otimes V_Y(\lambda_j)\right)|_G	\cong 	 V_Y(\lambda)|_G \oplus V_Y(\lambda_{j+1})|_G.
\end{equation*}

Now if $1\leqslant j \leqslant n-2,$ then applying \eqref{[Dn<A2n-1] Weight restrictions for B} yields the restrictions $\lambda|_{T}=\varpi$ and $\lambda_{j+1}|_{T} = \varpi_{j+1}+\delta_{j,n-2}\varpi_n.$ Therefore Proposition \ref{[D]character_of_V_Y(lambda)_G_for_D} yields $V_Y(\lambda)|_G\cong V_G(\varpi) \oplus V_G(\varpi_{j-1}),$ while $V_Y(\lambda_{j+1})|_G \cong  V_G(\varpi_{j+1}+\delta_{j,n-2}\varpi_n)$ by \cite[Theorem 1, Table 1 $(\textnormal{I}_4,\textnormal{I}_5)$]{Se}. The assertion thus holds in this situation and so it remains to consider the case where $j=n-1.$ Here we have $\lambda|_T=\varpi,$ $\lambda_n|_T=2\varpi_n,$ and applying Corollary \ref{Corollary_on_V(lambda_n)} and Proposition \ref{[D]character_of_V_Y(lambda)_G_for_D} yields $
T(\varpi) \cong V_G(\varpi) \oplus V_G(\varpi_{n-1}+\varpi_n) \oplus V_G(2\varpi_{n-1}) \oplus V_G(2\varpi_n)$ as desired. The proof is complete.
\end{proof}
%-------------------------------------%  
%-------------------------------------%

%-------------------------------------%
%-----     chT(\varpi) for D     -----%
%-------------------------------------%
\begin{lem}\label{Another_chT(omega)_for_D}
Consider the dominant $T$-weight $\varpi=\varpi_1+\varpi_{n-1} \in X^+(T),$ and write $T(w)$ for the tensor product $V_G(\varpi_1)\otimes V_G( \varpi_{n-1}).$ Then $T(\varpi)$ is tilting and its formal character is given by 
$$
\ch T(\varpi)=\chi(\varpi) + \chi(\varpi_n).
$$
\end{lem}

\begin{proof}
By Lemma \ref{Dimension_of_various_irreducibles_in_type_D}, both $V_G(\varpi_1)$ and $V_G(\varpi_{n-1})$ are irreducible, thus showing that $T(\varpi)$ is tilted by Proposition \ref{Tensor product of modules with good filtrations}. Also $\ch T(\varpi)$ is independent of $p,$ so we may and shall assume $K$ has characteristic zero in the remainder of the proof. Applying Theorem \ref{Weyl's formula} and Table \ref{Various_dimensions_in_type_D} then yields 
$$
\dim T(\varpi)=\dim V_G(\varpi)+2^{n-1} > \dim V_G(\varpi),
$$
showing the existence of a second composition factor of $T(\varpi).$ Now $\Lambda^+(T(\varpi))=\{\varpi,\varpi_n\},$ and since $\m_{T(\varpi)}(\varpi)=\m_{V_G(\varpi)}(\varpi)=1,$ we deduce that $[T(\varpi),V_G(\varpi_n)]>0.$ Finally, an application of Table \ref{Various_dimensions_in_type_D} yields $\dim V_G(\varpi_n)=2^{n-1},$  thus completing the proof.
\end{proof}
%-------------------------------------%
%-------------------------------------%

%-------------------------------------%
%-----     chT(\varpi) for D     -----%
%-------------------------------------%
\begin{lem}\label{Another_other_chT(omega)_for_D}
Consider the dominant $T$-weight $\varpi=\varpi_1+2\varpi_n \in X^+(T),$ and write $T(w)$ for the tensor product $V_G(\varpi_1)\otimes V_G(2\varpi_n).$ Then $T(\varpi)$ is tilting and its formal character is given by 
$$
\ch T(\varpi)= \chi(\varpi)+\chi(\varpi_{n-1}+\varpi_n).
$$
\end{lem}

\begin{proof}
By Lemma \ref{Dimension_of_various_irreducibles_in_type_D}, both $V_G(\varpi_1)$ and $V_G(2\varpi_n)$ are irreducible $KG$-modules, and hence $T(\varpi)$ is tilting by Proposition \ref{Tensor product of modules with good filtrations}. Also $\ch T(\varpi)$ is independent of $p,$ so we may and shall assume $K$ has characteristic zero in the remainder of the proof. Observe that by Corollary \ref{Corollary_on_V(lambda_n)}, we have 
$$
(V_Y(\lambda_1)\otimes V_Y(\lambda_n))|_G \cong T(\varpi)\oplus V_G(\varpi_1)\otimes V_G(2\varpi_{n-1}),
$$
while on the other hand, Lemma \ref{Tensor_products_in_type_A}  yields $
\left(V_Y(\lambda_1)\otimes V_Y(\lambda_n)\right)|_G	\cong 	 V_Y(\lambda)|_G \oplus V_Y(\lambda_{n+1} )|_G.$ Now  Proposition \ref{[D]character_of_V_Y(lambda)_G_for_D} yields $V_Y(\lambda)|_G\cong V_G(\varpi) \oplus V_G(\varpi_1+2\varpi_{n-1}) + V_G(\varpi_{n-1}+\varpi_n),$ while applying  \cite[Theorem 1, Table 1 $(\textnormal{I}_5)$]{Se} gives $V_Y(\lambda_{n+1})|_G \cong  V_G(\varpi_{n+1}+ \varpi_n).$ Consequently we have 
$$
T(\varpi)\oplus V_G(\varpi_1)\otimes V_G(2\varpi_{n-1})\cong 	 V_G(\varpi) \oplus V_G(\varpi_1+2\varpi_{n-1}) \oplus  V_G(\varpi_{n-1}+\varpi_n )^2.
$$ 
Clearly $\dim T(\varpi)=\dim V_G(\varpi_1)\otimes V_G(2\varpi_{n-1}),$ as $\varpi_n$ and $\varpi_{n-1}$ are conjugate under the action of the graph automorphism of order $2$ of $G,$ from which one deduces the desired result.
\end{proof}
%-------------------------------------%
%-------------------------------------%

%%%%%%%%%%%%%%%%%%%%%%%%%%%%%%%%%%%%%%%%%%%%%%%%%%%%%%%%%%%%%%%%%%%%%%%%%%%%%%%%
\subsection{Various contributions to the truncated Jantzen $p$-sum formula} %%%%
%%%%%%%%%%%%%%%%%%%%%%%%%%%%%%%%%%%%%%%%%%%%%%%%%%%%%%%%%%%%%%%%%%%%%%%%%%%%%%%%

In this section, we compute certain contributions to the truncated Jantzen formula for some Weyl modules $V_G(\varpi),$ where $\varpi$ is as in the first column of Table \ref{Table_D}, starting by the case where $\varpi=2\varpi_1.$

%-----------------------------------------------%
%----- Type D: Sum formula for (2,0,...,0) -----%
%-----------------------------------------------%
\begin{prop}\label{Type_D:Sum_formula_for_(2,0,...,0)} 
Assume $\varpi=2\varpi_1,$ and write $\mu=0.$ Let $V_G(\varpi)=V^0\supsetneq V^1\supseteq \ldots \supseteq V^k \supseteq 0$ be a Jantzen filtration of $V_G(\varpi).$ Then 
$$
\nu^c_\mu(T_\varpi)=\nu_p(n)\ch L_G(\mu).
$$
\end{prop}

\begin{proof}
We proceed as in the proof of Proposition \ref{Type_B:Sum_formula_for_(2,0,...,0)}, starting by computing all contributions to $\nu_\mu^c(T_\varpi).$  Here the only dominant $T$-weights  $\nu\in X^+(T)$ such that  $\mu\preccurlyeq \nu \prec \varpi$ are $\varpi-\beta_1$ and $\mu$ itself. By \cite{Pr}, the former cannot afford the highest weight of a composition factor of $V_G(\varpi)$ and so Lemma \ref{no_contribution_for_certain_weights} shows that $\varpi-\beta_1$ cannot contribute to $\nu_\mu^c(T_\varpi).$ We now determine  all pairs $(\beta,r)\in I_\mu$ as in Theorem \ref{How_to_determine_contributions}. A straightforward computation yields 
$$
B_\mu = (n-1,\ldots,1,0),
$$
and since  $\varpi-\mu$ has support $\Pi,$ we get that  $\beta\in \{\varepsilon_1+\varepsilon_{\ell}: 2 \leqslant \ell \leqslant n\}$ by definition of $I_\mu.$ Recall from \cite[Planche II]{Bourbaki} that $\mathscr{W}$ acts by all permutations and even number of sign changes of the $\varepsilon_i.$ Also, one checks that for  $2 \leqslant \ell \leqslant n,$  we have 
$$
A_{\varepsilon_1+\varepsilon_\ell,r}= B_\mu +(2-r,\underbrace{0,\ldots,0}_{\ell-2},-r,\underbrace{0,\ldots,0}_{n-\ell}),~r\in \Z.
$$
Consequently $B_\mu \in \mathscr{W} \cdot A_{\varepsilon_1+\varepsilon_\ell,r}$ if and only if $\{| n+1-r| , |n-\ell-r| \}=\{n-1,n-\ell\}.$  We now study  each possibility separately. 
\begin{enumerate}
\item If $n+1-r=n-1,$ then $ r=2$ and hence $\nu_p(r)=0,$ since $p\neq 2$ by assumption. Hence the weight $\varpi-2(\varepsilon_1+\varepsilon_\ell)$ cannot contribute to $\nu^c_\mu(T_\varpi)$ in this situation. 
\item If $n+1-r=n-\ell,$ then $r=\ell +1,$ and $|n-\ell -r|= |n-2\ell -1|.$ One checks that the latter is equal to $n-\ell$ if and only if $\ell =n-1,$ in which case $r=n.$ 
\item If $n+1-r=-n+\ell,$ then $r=2n-\ell +1,$ and $|n-\ell -r| = n+1 \neq n-1.$ 
\item If $n+1-r=-n+1,$ then $r=2n,$ and $|n-\ell-r|= n + \ell \neq n-\ell.$
\end{enumerate} 
 
Consequently,   the only possible contribution to 
$\nu^c_\mu(T_\varpi)$ can only come from the situation where  $\ell=n-1$ and $r=n,$ in which case we have
$$
A_{\varepsilon_1+\varepsilon_{n-1},n}= (1,n-2,\ldots,2,1-n,0).
$$
Using the action of $\mathscr{W}$ described above, one deduces that $\mu=s_{\varepsilon_1+\varepsilon_n} s_{\varepsilon_1-\varepsilon_n}s_{\varepsilon_1-\varepsilon_{n-1}}\cdot (\varpi-r(\varepsilon_1+\varepsilon_{n-1})).$ Therefore $\nu^c_\mu(T_\varpi)=\nu_p(n)\chi _\mu (\mu),$ and since $\chi_\mu(\mu)=\ch L_G(\mu),$ the proof is complete.  
\end{proof}
%-----------------------------------------------%
%-----------------------------------------------%

Proceeding as in the proof of Proposition \ref{Type_D:Sum_formula_for_(2,0,...,0)} , we next compute the formal character $\nu^c_\mu(T_\varpi)$ in the situation where $\varpi=\varpi_1+\varpi_j$ for some $1<j<n-1,$ and $\mu=\varpi_{j-1}.$ Observe that the case $j=2$ was dealt with in \cite[Lemma 4.5.7]{McNinch}.

%-----------------------------------------------%
%----- Type D: Sum formula for (2,0,...,0) -----%
%-----------------------------------------------%
\begin{prop}\label{Type_D:Sum_formula_for_(1,0,...,0,1,0,...,0)} 
Assume $\varpi=\varpi_1+\varpi_j $ for some $2\leqslant j\leqslant n-2,$ and write $\mu =\varpi_{j-1}.$ Let $V_G(\varpi)=V^0\supsetneq V^1\supseteq \ldots \supseteq V^k \supseteq 0$ be a Jantzen filtration of $V_G(\varpi).$ Then 
$$
\nu^c_\mu(T_\varpi)=\nu_p(j+1)\ch L_G (\varpi_{j+1}+\delta_{j,n-2}\varpi_n  ) + \nu_p(2n-j+1)\ch L_G(\mu ).
$$
\end{prop}

\begin{proof} 
We refer the reader to \cite[Lemma 4.5.7]{McNinch} for a proof in the case where $j=2$ and hence assume $j>2$ in what follows. Write $\nu_1 =\varpi_{j+1}+\delta_{j,n-2}\varpi_n.$ For simplicity, we also assume $ j\leqslant n-3$ (so that $\nu_1=\varpi_{j+1}$) and leave the case where $j=n-2$ to the reader. Here the dominant $T$-weights $\nu\in X^+(T)$ satisfying $\mu  \preccurlyeq \nu \prec \varpi$ are $\nu_1,$ $\varpi_1+\varpi_{j-1},$ and $\mu .$  We next determine  all pairs $(\beta,r)\in I_{\nu_1}$ as in Theorem \ref{How_to_determine_contributions}. A straightforward computation yields  
$$
B_{\nu_1}=(n,n-1,\ldots,n-j,n-j-2,\ldots,1,0),
$$
and since $\varpi-\nu_1$ has support $\{\beta_1,\ldots,\beta_j\},$ we get that $\beta=\varepsilon_1-\varepsilon_{j+1}$ by definition of $I_{\nu_1}.$ Also, one easily checks that for $r\in \Z,$ we have 
$$
A_{r,\varepsilon_1-\varepsilon_{j+1}}= B_{\nu_1} +(1-r,\underbrace{0,\ldots,0}_{j-1},r-1,\underbrace{0,\ldots,0}_{n-j-1}). 
$$
We leave it to the reader to show that $B_{\nu_1}\in \mathscr{W}\cdot A_{r,\varepsilon_1-\varepsilon_{j+1}}$ if and only if $r=j+1,$ and then to deduce that $\mu $ contributes to $\nu^c_{\mu }(T_\varpi)$ by $\nu_p(j+1)\chi_{\mu }(\nu_1).$ 

Finally, we determine  all pairs $(\beta,r)\in I_{\mu }$ as in Theorem \ref{How_to_determine_contributions}. Here again, a straightforward computation yields  
$$
B_{\mu }=(n,n-1,\ldots,n-j+2,n-j,\ldots,1,0),
$$
and since $\varpi-\mu $ has support $\Pi,$ we get that $\beta\in \{\varepsilon_1+\varepsilon_\ell : 2\leqslant \ell \leqslant n\}$ by definition of $I_{\mu }.$ Also, one easily checks that $B_{\mu }\notin \mathscr{W}\cdot A_{\varepsilon_1+\varepsilon_\ell,r}$ if $\ell \neq j,$ and hence we assume $\ell = j$ in the remainder of the argument. Setting $\beta=\varepsilon_1+\varepsilon_j,$ we then get
$$
A_{\beta,r}= B_{\mu } + (1-r,\underbrace{0,\ldots,0}_{j-2}, -1-r,\underbrace{0,\ldots,0}_{n-j}), r\in \Z.
$$ 
As usual, a case by case analysis then shows that $B_{\mu }\in \mathscr{W}\cdot A_{\beta,r}$ if and only if $r=2n-j+1,$ from which one deduces that $\mu $ contributes to $\nu^c_{\mu }(T_\varpi)$ by $\nu_p(2n-j+1)\chi_{\mu }(\mu ).$ Consequently, we have 
$$
\nu_{\mu }^c(T_\varpi)=\nu_p(j+1)\chi_{\mu }(\nu_1)+\nu_p(2n-j+1)\chi_{\mu }(\mu ).
$$
 Now  $\chi_{\nu_1}(\mu )=\ch L_G(\mu )$ by Lemma \ref{Dimension_of_various_irreducibles_in_type_D}, while $\chi_{\mu }(\mu )=L_G(\mu )$ by definition, thus completing the proof.
\end{proof}
%-----------------------------------------------%
%-----------------------------------------------%

The next result is very similar to Proposition \ref{Type_D:Sum_formula_for_(1,0,...,0,1,0,...,0)}, the difference residing in the apparition of two composition factors having highest weights conjugate under the action of the graph automorphism of order $2$ of $G.$ Since the proof is fairly identical, the details are omitted.

%---------------------------------------------------%
%----- Type D: Sum formula for (1,0,...,0,1,1) -----%
%---------------------------------------------------%
\begin{prop}\label{Type_D:Sum_formula_for_(1,0,...,0,1,1)} 
Assume $\varpi=\varpi_1+\varpi_{n-1}+\varpi_n,$ and write $\mu=\varpi_{n-2}.$ Also consider a Jantzen filtration $V_G(\varpi)=V^0\supsetneq V^1\supseteq \ldots \supseteq V^k \supseteq 0$ of $V_G(\varpi).$ Then 
$$
\nu^c_\mu(T_\varpi)=\nu_p(n)\ch L_G(2\varpi_n) + \nu_p(n)\ch L_G(2\varpi_{n-1}) + \nu_p(n+2)\ch L_G(\mu).
$$
\end{prop} 
%---------------------------------------------------%
%---------------------------------------------------%

We conclude this section with a result describing the decomposition of $\chi_\mu(\varpi)$ in terms of characters of irreducibles for well-chosen $\varpi,\mu\in X^+(T).$ Even though the proof of the following proposition does not require the  use of the Jantzen $p$-sum formula, we record it here for completeness.

%---------------------------------------------------%
%----- Type D: Sum formula for (1,0,...,0,1,0) -----%
%---------------------------------------------------%
\begin{prop}\label{Type_D:Sum_formula_for_(1,0,...,0,1,0)} 
Let $\xi\in \{1,2\},$ and let $\varpi_\xi=\varpi_1+\xi\varpi_n.$ Also consider the dominant $T$-weight $\mu_\xi= \varpi_{n-1}+(\xi-1)\varpi_n.$  Then 
$$
\chi_{\mu_\xi}(\varpi_\xi)=\ch L_G(\varpi_\xi) + \epsilon_p(n) \ch L_G(\mu_\xi).
$$
\end{prop}

\begin{proof}
We give a proof of the Proposition in the situation where $\xi=1,$ and leave the case $\xi=2$ to the reader. Here the only dominant $T$-weight $\nu\in X^+(T)$ satisfying $\mu_1 \preccurlyeq \nu \prec \varpi$ is $\mu_1$ itself. Also, an application of Lemma \ref{Tensor_products_in_type_A} to the Levi subgroup of type $A_{n-1}$  corresponding to the simple roots $\beta_1,\ldots,\beta_{n-1}$ shows that   $[V_G(\varpi),L_G(\mu_1)]=\epsilon_p(n),$ thus completing the proof.
\end{proof}
%---------------------------------------------------%
%---------------------------------------------------%

%%%%%%%%%%%%%%%%%%%%%%%%%%%%
\subsection{Conclusion}  %%%
%%%%%%%%%%%%%%%%%%%%%%%%%%%%

We conclude by giving proofs of Theorem \ref{Main_Result_1},  Corollary  \ref{Introduction: corollary on the dimension (B)}, and Proposition \ref{Restrictions_of_irreducibles_of_type_A}, in the situation where $G$ is of type $D_n$  and $Y$ is of type $A_{2n+1} $ $(n\geqslant 3).$

%----------------------------------------------%
%------ Proof of Theorem 1 in the B-case ------%
%----------------------------------------------%
\begin{proof}[Proof of Theorem \textnormal{\ref{Main_Result_1}} (The $D_n$-case)]
We proceed exactly as in the $B_n$-case, dealing with each row of Table \ref{Table_D} separately. Since the argument is identical to the one given in the proof in the situation where $G$ is of type $B_n,$ we consistently refer the reader to the proof of the latter for more details.

\begin{enumerate}
\item If $\varpi=2\varpi_1,$ then replacing Lemma \ref{chT(omega)_for_B} by Lemma \ref{chT(omega)_for_D} and Proposition \ref{Type_B:Sum_formula_for_(2,0,...,0)} by Proposition \ref{Type_D:Sum_formula_for_(2,0,...,0)}   yields the desired assertion.
\item If $\varpi=\varpi_1+\varpi_j+\delta_{j,n-1}\varpi_n$ for some $1<j< n,$ then replacing Lemma \ref{chT(omega)_for_B} by Lemma \ref{chT(omega)_for_D}, Proposition \ref{Type_B:Sum_formula_for_(1,0,...,0,1,0,...,0)} by Proposition \ref{Type_D:Sum_formula_for_(1,0,...,0,1,0,...,0)}, and Lemma \ref{Dimension_of_various_irreducibles_in_type_B} by Lemma \ref{Dimension_of_various_irreducibles_in_type_D} yields the result.
\item If $\varpi=\varpi_1+\varpi_{n-1},$ then replacing Lemma \ref{chT(omega)_for_B} by Lemma \ref{Another_chT(omega)_for_D} and Proposition Proposition \ref{Type_B:Sum_formula_for_(1,0,...,0,1)} by Proposition \ref{Type_D:Sum_formula_for_(1,0,...,0,1,0)} yields the result.
\item If $\varpi=\varpi_1+2\varpi_n,$ then replacing Lemma \ref{Another_chT(omega)_for_B} by Lemma \ref{Another_other_chT(omega)_for_D} and Proposition Proposition \ref{Type_B:Sum_formula_for_(1,0,...,0,2)} by Proposition \ref{Type_D:Sum_formula_for_(1,0,...,0,1,0)} yields the result.
\end{enumerate}
Therefore the result holds for each dominant $T$-weight $\varpi$ as in the first column of Table \ref{Table_D}, thus completing the proof.
\end{proof}
%----------------------------------------------%
%----------------------------------------------% 

The proof of Corollary \ref{Introduction: corollary on the dimension (B)} in the $B_n$-case simply consists in computing the dimension of $V_G(\varpi)$ and $\rad(\varpi),$ this for each $\varpi$ appearing in the first column of Table \ref{Table_B}. The details are thus left to the reader. We conclude with a proof of Proposition \ref{Restrictions_of_irreducibles_of_type_A} in the case where $G$ is of type $B_n.$

%------------------------------------------------%
%----- Proof of Proposition 3 in the B-case -----%
%------------------------------------------------%
\begin{proof}[Proof of Proposition \textnormal{\ref{Restrictions_of_irreducibles_of_type_A}} (The $B_n$-case)]
We give a proof in the situation where $1\leqslant j\leqslant n-1,$ and omit  the other cases, as they can be dealt with in a similar fashion. First observe that by Lemma \ref{Tensor_products_in_type_A}, we have
$$
\ch V|_G=\ch V_Y(\lambda)|_G-\epsilon_p(j+1)\ch L_Y(\lambda_{j+1})|_G.
$$
Also, an application of Proposition \ref{Type_B:restriction_in_characteristic_zero_0<j<n+1} yields $\ch V_Y(\lambda)|_G = \chi(\varpi)+\ch L_G(\varpi_{j-1}),$ while by \cite[Theorem 1, Table 1 $(\textnormal{I}_2,\textnormal{I}_3)$]{Se}, we have $\ch L_Y(\lambda_{j+1})|_G= \ch L_G(\varpi_{j+1}+\delta_{j,n-1}\varpi_n).$ Therefore
$$
\ch V|_G=\chi(\varpi)+ \ch L_G(\varpi_{j-1}) -\epsilon_p(j+1)\ch L_G(\varpi_{j+1}+\delta_{j,n-1}\varpi_n).
$$
Now $\chi(\varpi)= \ch L_G(\varpi)+\epsilon_p(j+1)\ch L_G(\varpi_{j+1}+ \delta_{j,n-1}\varpi_n) + \epsilon_p(2n-j+2)\ch L_G(\varpi_{j-1}) $ by  Theorem \ref{Main_Result_1}, and so the result follows.
\end{proof}
%------------------------------------------------%
%------------------------------------------------% 

%%%%%%%%%%%%%%%%%%%%%%%%%%%%%%%%%%%%%
%%%%     Acknowledgements     %%%%%%%
%%%%%%%%%%%%%%%%%%%%%%%%%%%%%%%%%%%%%
\section*{Acknowledgements}     %%%%%
%%%%%%%%%%%%%%%%%%%%%%%%%%%%%%%%%%%%%

I would like to express my deepest thanks to my Ph.D. advisor Professor Donna M. Testerman, for her constant support and her precious guidance during my doctoral  studies, from which this paper is partially drawn. I would also like to extend my appreciation to Professors Timothy C. Burness and Frank L\"ubeck, for their very helpful comments and suggestions on the earlier versions of this paper.

%%%%%%%%%%%%%%%%%%%%%%%%%%%%%%%%%%%%%%
\bibliographystyle{amsalpha}     %%%%%
%%%%%%%%%%%%%%%%%%%%%%%%%%%%%%%%%%%%%%

\bibliography{References}

\end{document}